\documentclass[12pt]{amsart}
\usepackage{amssymb,latexsym,amsmath,amsthm,amsfonts, enumerate}
\usepackage{color}
\usepackage[all]{xy}
\usepackage{tikz}
\usepackage{tikz-cd}
\usetikzlibrary{positioning,shapes,shadows,arrows,snakes,matrix,patterns,calc}
\usepackage[colorlinks=true,pagebackref,hyperindex]{hyperref}
\usepackage{pgfplots}
\usepgfplotslibrary{fillbetween}
\usepackage{makecell}

\advance \topmargin by -\headheight
\advance \topmargin by -\headsep
\evensidemargin \oddsidemargin
\numberwithin{equation}{section}

\newtheorem{theorem}{Theorem}[section]
\newtheorem{proposition}[theorem]{Proposition}

\newtheorem{lemma}[theorem]{Lemma}

\theoremstyle{definition}

\newtheorem{example}[theorem]{Example}
\newtheorem{definition}[theorem]{Definition}

\let\oldmarginpar\marginpar
\renewcommand\marginpar[1]{\-\oldmarginpar[\raggedleft\small\sf
#1]{\raggedright\small\sf #1}}

\newcommand{\ZZ}{\mathbb{Z}}
\newcommand{\CC}{\mathbb{C}}

\newcommand{\cA}{\mathcal{A}}

\newcommand{\cT}{\mathcal{T}}
\newcommand{\D}{\mathrm{Dom}}

\newcommand{\Hom}{\mathrm{Hom}}

\def \v{\mathbf{v}}

\newcommand{\myE}{\mathbf{E}}

\newcommand{\mytildeF}{\tilde{\mathcal{F}}}

\newcommand{\im}{\mathrm{im}}

\newcommand{\rank}{\mathrm{rank}}

\makeatletter
\newcommand*{\doublerightarrow}[2]{\mathrel{
  \settowidth{\@tempdima}{$\scriptstyle#1$}
  \settowidth{\@tempdimb}{$\scriptstyle#2$}
  \ifdim\@tempdimb>\@tempdima \@tempdima=\@tempdimb\fi
  \mathop{\vcenter{
    \offinterlineskip\ialign{\hbox to\dimexpr\@tempdima+2em{##}\cr
    \rightarrowfill\cr\noalign{\kern.2ex}
    \rightarrowfill\cr}}}\limits^{\!#1}_{\!#2}}}
\newcommand*{\doubleleftarrow}[2]{\mathrel{
  \settowidth{\@tempdima}{$\scriptstyle#1$}
  \settowidth{\@tempdimb}{$\scriptstyle#2$}
  \ifdim\@tempdimb>\@tempdima \@tempdima=\@tempdimb\fi
  \mathop{\vcenter{
    \offinterlineskip\ialign{\hbox to\dimexpr\@tempdima+2em{##}\cr
    \leftarrowfill\cr\noalign{\kern.2ex}
    \leftarrowfill\cr}}}\limits^{\!#1}_{\!#2}}}
\newcommand*{\triplerightarrow}[1]{\mathrel{
  \settowidth{\@tempdima}{$\scriptstyle#1$}
  \mathop{\vcenter{
    \offinterlineskip\ialign{\hbox to\dimexpr\@tempdima+2em{##}\cr
    \rightarrowfill\cr\noalign{\kern.2ex}
    \rightarrowfill\cr\noalign{\kern.2ex}
    \rightarrowfill\cr}}}\limits^{\!#1}}}
\makeatother

\begin{document}

\title[Nakajima's quiver varieties and triangular bases]
{Nakajima's quiver varieties and triangular bases of bipartite cluster algebras}
\author{Li Li}
\address{Department of Mathematics
and Statistics,
Oakland University, 
Rochester, MI 48309-4479, USA 
}
\email{li2345@oakland.edu}


\begin{abstract}
Berenstein and Zelevinsky introduced quantum cluster algebras \cite{BZ1} and the triangular bases \cite{BZ2}. The support conjecture proposed in \cite{LLRZ}, which asserts that the support of each triangular basis element for a rank-2 cluster algebra is bounded by an explicitly described region, was established in \cite{L} for skew-symmetric rank-2 cluster algebras. In this paper we extend this result  by proving a bound on the support of each triangular basis element for bipartite cluster algebras. 
\end{abstract}

\maketitle
\tableofcontents

\section{Introduction}
Cluster algebras and quantum cluster algebras were introduced by Fomin-Zelevinsky \cite{fz-ClusterI} and Berenstein-Zelevinsky \cite{BZ1}, respectively. A primary objective behind the inception of  quantum cluster algebras is to understand good bases arising from the representation theory of certain non-associative algebras. Mimicking the construction of the dual canonical basis from a PBW basis \cite{Leclerc}, Berenstein and Zelevinsky \cite{BZ2} constructed the triangular bases for acyclic quantum cluster algebras. These triangular bases exhibit many desirable properties. It particular, it is worth mentioning the work by Qin \cite{Qin, Qin1, Qin2}, where he has constructed triangular bases for injective-reachable cluster algebras using a different approach, and showed that his triangular bases  coincide with the bases by Berenstein and Zelevinsky for the seeds associated with acyclic quivers.

In  \cite{LLRZ}, Lee, Rupel, Zelevinsky and the author proposed a  conjecture of a precisely described region to constrain the support of each triangular basis element of any rank-2 quantum cluster algebra.
The proof for the skew-symmetric rank-2 case was established in \cite{L} using Nakajima's quiver varieties. In this work, we prove a similar result for any bipartite cluster algebra. 

The definition of quantum cluster algebra and (Berenstein-Zelevinsky's) triangular basis will be reviewed in \S \ref{section:review qCA}. Our focus here is on bipartite quantum cluster algebras, that is, those come from bipartite quiver. Additionally, for the sake of clarity and simplicity, we confine our study to quantum cluster algebras with principal quantization.

Consider a rank-$n$ bipartite cluster algebra with principal quantization determined by an $n\times n$ skew-symmetric matrix $B=[b_{ij}]$. Let the initial seed be $(\Lambda,\tilde{B},\tilde{X})$, where $$\tilde{B}=\begin{bmatrix}B\\ {\bf I}_n\end{bmatrix},\quad
\Lambda=\begin{bmatrix}0&-{\bf I}_n\\ {\bf I}_n&-B\end{bmatrix},\quad
\tilde{X}=(X_1,\dots,X_{2n})$$ 
where ${\bf I}_n$ the $n\times n$ identity matrix. In this setup, 
the first $n$ variables of $\tilde{X}$ are called cluster variables, while the last $n$ variables are called frozen variables. 

For $c\in\mathbb{R}$, we denote $[c]_+=\max(c,0)$;  the vector $b_k$ is the $k$-th column  of $\tilde{B}$; for a vector $b$, we use $[b]_+$ to denote the vector obtained from $b$ by applying $[\; ]_+$ to each coordinate. 

The  triangular basis $\{C_a\}_{a\in \mathbb{Z}^{2n}}$ is of the following form: 
\begin{equation}\label{eq:C=e}
C_a=\sum_v e_v X^{a+\tilde{B}v}, \quad e_v\in\mathbb{Z}[\v^\pm].
\end{equation}
We define $\deg(e_v)$ to be largest exponent of $\v$ in $e_v$. 
Define a quiver $\mathcal{Q}^{\rm op}=(I,\Omega)$ with vertex set $I=\{1,\dots,n\}=I_0\sqcup I_1$ (where vertices in $I_0$ are sinks, vertices in $I_1$ are sources) and arrow set $\Omega$ determined by $B$ such that there are $b_{ij}$ arrows $h:j\to i$ whenever $b_{ij}>0$. We denote the source and target of $h$ by $s(h)$ and $t(h)$, respectively.  
See \S2 for details.

\begin{theorem}\label{main theorem}
Let $a=(a_1,\dots,a_{2n})\in\mathbb{Z}^{2n}$. Each $e_v$ in $C_a$ as in \eqref{eq:C=e} is symmetric and unimodal, and
 $$\deg(e_v)  \le  f(v) := -\sum_{i=1}^n (a_i+v_i)v_i+\sum_{h\in\Omega}v_{s(h)}v_{t(h)}.$$
In particular,  $e_v\neq0$ only if $f(v)\ge0$.
\end{theorem}

Together with \eqref{nonempty F}, the support 
$${\rm Supp}(C_a):=\{v\in\mathbb{Z}_{\ge0}^n \; | \; e_v\neq0\}$$
satisfies the following condition:
$$f(v)\ge0, \quad
0\le v_\alpha\le [-a_\alpha]_+ \; (\forall \alpha\in I_0),\; \textrm{ and }\; 0\le v_\beta\le [a_\beta]_+ +\sum_{h:\ s(h)=\beta} v_{t(h)} \; (\forall \beta\in I_1).
$$
Note that not every integer point satisfying the above condition are in ${\rm Supp}(C_a)$. In an ongoing project, the author is looking for a precise bound of ${\rm Supp}(C_a)$ to full generalize the results in  \cite{L}.  

The proof of the theorem follows the idea presented in \cite{L}. While a significant portion of the methodology used in \cite{L} can be extended to the bipartite case, many adjustments are necessary,  and the proofs and computations become notably more intricate. Certain results are already established in the work of Nakajima \cite{Nakajima} and Qin \cite{Qin};  however, our approach endeavors to provide a self-contained treatment, with minimal reliance on representation theory and predominantly using elementary algebraic geometry, in alignment with the approach adopted in \cite{L}. 

\smallskip

The paper is organized as follows. 
In \S2 we revisit the definition and properties of quantum cluster algebra and triangular basis.
In \S3 we review some facts on Nakajima's quiver varieties.
In \S4 we prove an algebraic version of the transversal slice theorem, and show that the local systems appeared in the BBDG decomposition theorem are trivial. The BBDG Decomposition theorem \cite{BBD} serves as the primary advanced tool used in this paper. It asserts that, for a proper map of algebraic varieties $f:X\to Y$, the pushforward of the intersection cohomology complex $IC_X$ can be decomposed into a direct sum of $IC_V(L)[d]$ for a collection of irreducible subvarieties $V\subseteq Y$, local systems $L$ on (an open dense subset of) $V$, and integers $d$. While the local systems are often trivial, as is the case in our scenario, the proof is not always straightforward. Here, we present a proof using the Algebraic Transversal Slice Theorem. 
In \S5 we identify the dual canonical basis and triangular basis in our specific case.
In \S6 we prove the main theorem and provide some examples. 

\smallskip

{\bf Acknowledgments.} The author thanks the organizers, Yiqiang Li and Changlong Zhong, of the AMS Special Session held at the University of Buffalo in September 2023, during which numerous inspiring conversations took place.  
Computer calculations were performed using SageMath \cite{Sage}.

\section{Quantum cluster algebras and triangular bases}\label{section:review qCA}

In this section we recall the definition and facts of quantum cluster algebras with principal quantization. See \cite{BZ1,BZ2} for details. 

\subsection{Quantum cluster algebras}
Let $n$ be a positive integer and $m=2n$. Let $B$ be a skew-symmetrizable integer $n\times n$ matrix, that is, there exists a diagonal matrix $D$ such that $DB$ is a skew-symmetric matrix. 
Denote the $n\times n$ identity matrix by ${\bf I}_n$, or simply by ${\bf I}$ if $n$ is clear from the context.
Let $\tilde{B}:=\begin{bmatrix}B\\ {\bf I}_n\end{bmatrix}$ where ${\bf I}_n$ the $n\times n$ identity matrix. 
Let $\Lambda:=\begin{bmatrix}0&-D\\ D&-DB\end{bmatrix}$, viewed as a bilinear form by defining $\Lambda({\bf u}, {\bf u}')={\bf u}^T\Lambda{\bf u}'$ for column vectors ${\bf u}, {\bf u}'\in\mathbb{Z}^{m}$. 
Let the quantum torus $\cT$ be the $\ZZ[\v^{\pm 1}]$-algebra with basis $\{X^{\bf e} : {\bf e}\in\mathbb{Z}^{m}\}$, and the multiplication given by $X^{\bf e}X^{{\bf e}'}=\v^{\Lambda({\bf e},{\bf e}')}X^{{\bf e}+{\bf e}'}$ for ${\bf e},{\bf e}'\in\mathbb{Z}^{m}$. 
Let $\mathcal{F}$ be the skew-field of fractions of $\cT$.
Let $e_1,\dots,e_m\in\mathbb{Z}^m$ be the standard basis, that is, $e_i$ has 1 in the $i$-th coordinate and 0 elsewhere. 
Let $\tilde{X}=\{X_1,\dots,X_{m}\}$ where $X_i=X^{e_i}$  are called (initial) cluster variables if $1\le i\le n$, or frozen variables if $n+1\le i\le m$.
The \emph{bar-involution} is the $\mathbb{Z}$-linear anti-automorphism of $\cT$ determined by 
$\overline{\v^iX^{\bf e}}=\v^{-i}X^{\bf e}$ for ${\bf e}\in\mathbb{Z}^{m}$. Note that $\overline{XY}=\overline{Y}\; \overline{X}$ for $X,Y\in \cT$. 
An element in $\cT$ which is invariant under the bar-involution is said to be \emph{bar-invariant}.

Fix an $n$-regular tree $\mathbb{T}_n$ with root $t_0$ and label each edge by a number in $\{1,\dots,n\}$ such that any two edges sharing an endpoint are labelled differently. Attach to $t_0$ the tuple $(\Lambda,\tilde{B},\tilde{X})$ as above, called the initial seed. Given an edge  joining $t$ and $t'$ with label $k$, we denote $t'=\mu_k(t)$ (and $t=\mu_k(t')$); if the seed $(\Lambda,\tilde{B},\tilde{X})$ at $t$ is constructed, we construct a seed $(\Lambda',\tilde{B}',\tilde{X}')$ attached to $t'$ by mutation, where $\tilde{B}'=[b'_{ij}]$ is obtained from $\tilde{B}=[b_{ij}]$ by
$$b'_{ij}=
\begin{cases}
&-b_{ij}, \textrm{ if $i=k$ or $j=k$}; \\
&b_{ij} + [b_{ik}]_+[b_{kj}]_+ - [-b_{ik}]_+[-b_{kj}]_+, \textrm{ otherwise.}\\
\end{cases}
$$
and $\tilde{X}'$ is obtained from $\tilde{X}$ by replacing $X_k$ with $X'_k=X^{-e_k+[b_k]_+} + X^{-e_k+[-b_k]_+}$. 
The bilinear form $\Lambda'$ is determined by 
$$
\begin{cases}
&\Lambda'(e_i,e_j)=\Lambda(e_i,e_j), \textrm{ for $i,j\neq k$;}\\
&\Lambda'(e_k,e_j)=\Lambda(-e_k+[b_k]_+,e_j), \textrm{ for $j\neq k$.}\\ 
\end{cases}
$$
Denote $\mathbb{Z}P[\v\pm]:=\mathbb{Z}[\v^\pm][X_{n+1}^\pm,\dots,X_{m}^\pm]$.  
The \emph{quantum cluster algebra} $\cA$  is the $\mathbb{Z}P[\v^{\pm1}]$-subalgebra of $\mathcal{F}$ generated by all cluster variables. 
%

\subsection{Triangular bases} 
Assume that the seed $(\Lambda,\tilde{B},\tilde{X})$ is acyclic, that is, there exists a linear order $\lhd$ on $\{1,\dots,n\}$ such that $b_{ij}\le 0$ if $i\lhd j$. Define $r(a)=\sum_{k=1}^n[-a_k]_+$ for $a\in\mathbb{Z}^{m}$ and define a partial order $\prec$ on $\mathbb{Z}^{m}$ by $a'\prec a\Leftrightarrow r(a')<r(a)$. 

The construction of the triangular basis starts with the standard monomial basis $\{E_a\,:\, a\in\ZZ^{m}\}$.
Let $X'_k=\mu_k(X_k)$.
For every $a=(a_1,\dots,a_{m})$, the \emph{standard monomial} $E_a$ is defined as 
 \begin{equation}\label{eq:Ma}
  E_{a} = \v^{v(a)}
  X^{\sum_{i=n+1}^{m}a_ie_i+\sum_{j=1}^n[a_j]_+ e_j} \prod_{1\le k\le n}^\lhd (X_k')^{[-a_k]_+}\ .
 \end{equation}
 where $\prod_{1\le k\le n}^\lhd$ means taking the product in increasing order with respect to $\lhd$,
 and $v({a})\in\mathbb{Z}$ is determined by the condition that the leading term of $E_a$,  which is the term obtained by 
 replacing $X_k'$ by $X^{-e_k-[b_k]_+}$ in \eqref{eq:Ma},  
 is bar-invariant. (In \cite{BZ2}, $X^{-e_k+[b_k]_+}$ is used, but as they remarked, it can be replaced by $X^{-e_k-[b_k]_+}$ and the value of $v(a)$ remains the same. We use the latter, because then $a$ gives the $g$-vector of $E_a$. This is also the choice in \cite{Qin}.)
  
 It is known 
 that $\{E_a\}$ is a $\ZZ[\v^{\pm 1}]$-basis of the cluster algebra  $\cA$.
However, these $E_a$ are not bar-invariant in general and do not contain all the cluster monomials, and they are inherently dependent on the choice of an initial cluster. These drawbacks provide a motivation to consider the triangular basis
(which is introduced in \cite{BZ2} and recalled below) constructed from the standard monomial basis with a
built-in bar-invariance property.

\begin{definition}\cite{BZ2}
The triangular basis $\{C_a\,:\,{a}\in\ZZ^{2n}\}$ is the unique collection of elements in $\cA(\tilde{X},\tilde{B})$ satisfying:

 \begin{enumerate}
  \item[] (P1) Each $C_a$ is bar-invariant; 
  \item[] (P2) For each $a$,
  $\displaystyle  C_a - E_a \in \bigoplus_{a'} \mathbf{v}\ZZ[\mathbf{v}]E_{a'}.$
  \end{enumerate}
Moreover, those $a'$ satisfy $a'\prec a$.
\end{definition}
It is shown in \cite{BZ2} that
the triangular basis exists uniquely, does not depend on the choice of an initial (acyclic) seed, and contains all cluster monomials associated to acyclic seeds.

\medskip

\noindent{\bf Note: In the rest of the paper we assume $B$ is skew-symmetric and take $D=I_n$.}

\section{Nakajima's graded quiver varieties}

In this section we recall Nakajima's graded quiver varieties. 
The references for this section are \cite{Nakajima, Qin}. There is some inconsistency in these references. In this paper, 
We will start with a quiver $\mathcal{Q}^{op}$ (pardon the notation), since it is used essentially in the paper, for Nakajima's graded quiver varieties. Its opposite quiver $\mathcal{Q}$ is to match the exchange matrix $B$ for quantum cluster algebras in the usual convention. 

Let $\mathcal{Q}^{\rm op}$ be a bipartite quiver with vertex set $I=\{1,\dots,n\}=I_0\sqcup I_1$ such that vertices in $I_0$ (resp.~$I_1$) are sinks (resp.~sources), and let $\Omega=\{h_1,\dots,h_r\}$ be the set of arrows in $\mathcal{Q}^{\rm op}$, and denote $\alpha_j := t(h_j)$, $\beta_j := s(h_j)$, that is, $h_j$ is an arrow $\beta_j\to \alpha_j$. 
Let $\widetilde{\mathcal{Q}^{\rm op}}$ be the decorated quiver obtained from $\mathcal{Q}^{\rm op}$ by adding a new vertex $i'$ and an arrow $i'\to i$ if $i\in I_0$ (resp.~an arrow $i\to i'$ if $i\in I_1$).

Following the notation in \cite[\S5]{Nakajima}, denote dimension vectors $w=(w_i,w'_i)_{i\in I}$, $v=(v_i)_{i\in I}$, where all the components are nonnegative integers.

Fix $\mathbb{C}$-vector spaces $W_i,W'_i, V_i$ with dimensions 
$w_i,w'_i, v_i$ for $i\in I$, respectively. Define
$$W := \bigoplus_{i\in I}(W_i\oplus W_i'), \quad V := \bigoplus_{i\in I} V_i.$$
A quiver representation of $\widetilde{\mathcal{Q}^{\rm op}}$ consists linear maps $x_i: W'_i\to W_i$ ($i\in I_0$), $x_i: W_i\to W_i'$ ($i\in I_1$),  and $y_h:W_{s(h)}\to W_{t(h)}$ ($h\in \Omega$). Below is an example with $I_0=\{3\}$ and $I_1=\{1,2\}$: 
\begin{equation}\label{WWWW}
\xymatrix{
W_1\ar[d]_-{x_1} \ar@<-.5ex>[rd]^(.35){y_1} \ar@<.5ex>[rd]_(.35){\phantom{X} y_2} 
 & W'_3\ar[d]^-(.3){x_3} &W_2\ar[d]^{x_2} \ar@<-.5ex>[ld]^-(.35){y_4} \ar@<.5ex>[ld]_-(.3){\phantom{XX} y_3}\\
W_1' & W_3 &W_2'\\
}
\end{equation}
\begin{definition}\label{def:Ai}
For each sink $\alpha\in I_0$, let $r_1<\cdots<r_s$ be the indices such that $h_{r_1},\dots,h_{r_s}$ are all the arrows with target $\alpha$, and define a linear map
$$
\aligned
A_\alpha := x_\alpha+\sum_{h:\ t(h)=\alpha} y_h: &\; W_\alpha' \oplus \bigoplus_{h:\ t(h)=\alpha} W_{s(h)}\to W_\alpha, \\
&(a,b_1,\dots,b_s)\mapsto x_\alpha(a)+y_{r_1}(b_1)+\cdots+y_{r_s}(b_r)\\
\endaligned
$$
For each source $\beta\in I_1$, let $r'_1<\cdots<r'_t$ be the indices such that $h_{r'_1},\dots,h_{r'_t}$ are all the arrows with source $\beta$, and define a linear map
$$
\aligned
A_\beta := x_\beta \oplus \bigoplus_{h:\  s(h)=\beta} y_h: &\; W_\beta\to W_\beta'\oplus \bigoplus_{h:\ s(h)=\beta} W_{t(h)},\\ & b\mapsto (x_\beta(b), y_{r'_1}(b),\dots,y_{r'_t}(b))
\endaligned
$$
\end{definition}
Note that  $s,r_1,\dots,r_s$ depend on $\alpha$ so, strictly speaking, we should write them as $s(\alpha)$, $r_1(\alpha)$, etc., but we omit  $\alpha$ for simplicity of notation; similarly, $t,r'_1,\dots,r'_t$ depend on $\beta$. 

Alternatively, we can express these maps by matrices. Fix bases for $W_i$ and $W'_i$ , and by abuse of notation we use  $x_i$ (resp.~$y_h$) to denote the matrix representing the corresponding linear map. 
Then the above $A_\alpha$ and $A_\beta$ are represented by the following block matrices: 
\begin{equation}\label{AB}
\aligned
&
A_\alpha =
[x_\alpha, y_{r_1},\dots,y_{r_s}]_\text{hor}
:=[ x_\alpha | y_{r_1} | y_{r_2} |\hdots |y_{r_s} ], \\
&A_\beta=[x_\beta,y_{r'_1},y_{r'_2} \dots,y_{(r'_t)} ]_\text{vert}
:=
\begin{bmatrix} x_\beta \\ y_{r'_1}\\ y_{r'_2}\\ \vdots\\ y_{r'_t}
\end{bmatrix}.
\\
\endaligned
\end{equation}

Below is the list of sizes of various matrices, where $\alpha\in I_0$, $\beta=I_1$:

\begin{center}
 \begin{tabular}{|c|c|c|c|c|} 
 \hline
 $x_\alpha$ & $x_\beta$ & $t_h$ & $A_\alpha$ & $A_\beta$\\
 \hline
 $w_\alpha\times w'_\alpha$ & $w_\beta'\times w_\beta$ & $w_{t(h)}\times w_{s(h)}$ & $\displaystyle w_\alpha\times (w_\alpha'+\sum_{h:\;  t(h)=\alpha } w_{s(h)})$ & 
$\displaystyle (w'_\beta+\sum_{h:\;  s(h)=\beta } w_{t(h)}) \times w_\beta$\\
 \hline
\end{tabular}
\end{center}

\smallskip

For a vector space $E$ of dimension $e$, denote by $Gr(d,E)$ or $Gr(d,e)$ the Grassmannian space parametrizing all $d$-dimensional linear subspaces of $E$.
Let $S_{Gr(d,e)}$ be the tautological subbundle (i.e., the universal subbundle) on the Grassmannian variety $Gr(d,e)$.

\begin{definition}\cite[\S4]{Nakajima} 
Let dimension vectors $w=(w_i,w'_i)_{i\in I}$, $v=(v_i)_{i\in I}$, where all the components are nonnegative integers.

(a) The variety $\mathbf{E}_w$ is defined as the space of quiver representations of the quiver $\widetilde{\mathcal{Q}^{\rm op}}$:
$$\aligned
\mathbf{E}_w
& := \bigoplus_{\alpha\in I_0}\Hom(W'_\alpha, W_\alpha) \oplus
\bigoplus_{\beta\in I_1} \Hom(W_\beta, W'_\beta) \oplus \bigoplus_{h\in\Omega} \Hom(W_{s(h)},W_{t(h)}) \\
&\cong \CC^{\sum_{i\in I}w_iw_i'+\sum_{h\in\Omega} w_{s(h)}w_{t(h)} }.
\endaligned
$$ 
We denote elements of $\mathbf{E}_w$ as
$(x_i,y_h)_{i\in I,\,  h\in\Omega}$.

(b) The projective variety $\mathcal{F}_{v,w}$ is a closed subvariety of $$\prod_{\alpha\in I_0} Gr(v_\alpha,W_\alpha)\times \prod_{\beta\in I_1} Gr\Big(v_\beta,W'_\beta \oplus \bigoplus_{h:\ s(h)=\beta} W_{t(h)}\Big)$$ 
defined as
$$\aligned
\mathcal{F}_{v,w} := \big\{(X_i)_{i\in I}\; |\; 
& \dim X_i=v_i\, (\forall i\in I), \;
X_\alpha\subseteq W_\alpha (\forall \alpha\in I_0),\\ 
& X_\beta\subseteq W_\beta'\oplus \bigoplus_{h:\ s(h)=\beta}  X_{t(h)} (\forall \beta\in I_1)\big\}.
\endaligned
$$

(c) \emph{Nakajima's nonsingular graded quiver variety} $\tilde{\mathcal{F}}_{v,w}$ is given by
$$\aligned
\tilde{\mathcal{F}}_{v,w}
 := \big \{(x_i,y_h,X_i)_{i\in I, h\in\Omega}\;|\; 
& 
(x_i,y_h)_{i\in I, h\in\Omega} \in \mathbf{E}_w, 
(X_i)_{i\in I}\in\mathcal{F}_{v,w},\\ 
& {\im } A_\alpha\subseteq X_\alpha \; (\forall \alpha\in I_0),\;  {\im } A_\beta\subseteq X_\beta \; (\forall \beta\in I_1) \big\}.
 \endaligned
$$
(Note that $A_\alpha$, $A_\beta$ are determined by $(x_i,y_h)_{i\in I, h\in\Omega}$.)

(d) \emph{Nakajima's affine graded quiver variety} is 
$$
\aligned
\mathbf{E}_{v,w}
& := \{(x_i,y_h)_{i\in I, h\in \Omega}\in\mathbf{E}_w\, |\, 
\mathrm{rank}A_i\le v_i \; (\forall i\in I)\}.
\endaligned
$$
\end{definition}

Nakajima showed in \cite{Nakajima} that the above definition of
Nakajima's nonsingular (resp.~affine) graded quiver variety is equivalent to the original definition given by the GIT quotient (resp.~algebro-geometric quotient). 
The following statement is also proved by Nakajima. 
\begin{lemma}\label{EF}
Given dimension vectors $w=(w_i,w'_i)_{i\in I}\in\mathbb{Z}_{\ge0}^{2n}$, $v=(v_i)_{i\in I}\in\mathbb{Z}_{\ge0}^n$. 

{\rm(a)} The variety $\myE_{v,w}$ contains the origin, so is nonempty.   
The variety $\mathcal{F}_{v,w}$ (thus $\tilde{\mathcal{F}}_{v,w}$)  is nonempty if and only if 
\begin{equation}\label{nonempty F}
0\le v_\alpha\le w_\alpha \; (\forall \alpha\in I_0),\; \textrm{ and }\; 0\le v_\beta\le w_\beta'+\sum_{h:\ s(h)=\beta} v_{t(h)} \; (\forall \beta\in I_1).
\end{equation}

{\rm(b)}  If \eqref{nonempty F} holds, then  the natural projection $$
\pi': \mytildeF_{v,w}\to \mathcal{F}_{v,w}, \; (x_i,y_h,X_i)_{i\in I, h\in\Omega} \mapsto (X_i)_{i\in I}
$$ 
gives a vector bundle of rank  $\sum_{\alpha\in I_0} v_\alpha w'_\alpha+\sum_{\beta\in I_1}v_\beta w_\beta$. Meanwhile, the natural projection  
\begin{equation}\label{pi}
\pi: \tilde{\mathcal{F}}_{v,w}\to \mathbf{E}_{v,w},\quad (x_i,y_h,X_i)_{i\in I, h\in\Omega} \mapsto (x_i,y_h)_{i\in I, h\in\Omega}
\end{equation}
has the zero fiber $\pi^{-1}(0)\cong \mathcal{F}_{v,w}$.
\end{lemma}
\begin{proof}
Part (a) and the second statement of (b) are obvious. For the first statement of (b), following a similar proof of \cite[Lemma 4.3]{L}, we have  
\begin{equation}\label{eq:tildeF is vector bundle}
\mytildeF_{v,w} \cong \iota^*\Big(\bigoplus_{\alpha\in I_0} S_{Gr(v_\alpha,W_\alpha)}^{\oplus w'_\alpha}
\oplus
\bigoplus_{\beta\in I_1} S_{Gr(v_\beta,W'_\beta\oplus \bigoplus_{h:\ s(h)=\beta} {W_{t(h)}})}^{\oplus w_\beta}\Big)
\end{equation}
 where $\iota$ is the natural embedding
$$\iota:\mathcal{F}_{v,w}\to \prod_{\alpha\in I_0} Gr(v_\alpha,W_\alpha) \times \prod_{\beta\in I_1} Gr(v_\beta,W'_\beta\oplus \bigoplus_{h:\ s(h)=\beta} {W_{t(h)}}).$$  It follows that $\mytildeF_{v,w}$ is a vector bundle over $\mathcal{F}_{v,w}$ of rank $\sum_{\alpha\in I_0} v_\alpha w'_\alpha+\sum_{\beta\in I_1}v_\beta w_\beta$.
\end{proof}

Consider the following stratification $\myE_w=\bigcup_{v} \myE^\circ_{v,w}$ where
\begin{equation}\label{Ust}
\mathbf{E}^\circ_{v,w}
:= \{(x_i,y_h)_{i\in I, h\in \Omega}\in\mathbf{E}_w\, |\, 
\mathrm{rank}A_i = v_i \; (\forall i\in I)\}.
\end{equation}

For any two vectors $u=(u_1,\dots,u_n)$ and $v=(v_1,\dots,v_n)$, we denote $u\le v$ if $u_i\le v_i$ for all $i$. 
We say that $j$ is adjacent to $i$ if there is an arrow $h\in\Omega$ of the form $i\to j$ or $i\leftarrow j$; we denote this arrow as $h:i-j$ whenever the direction of the arrow is irrelevant. 
Recall the $q$-analog $C_q$ of the Cartan matrix in \cite[(3.2)]{Nakajima}:
\begin{equation}\label{C_qv}
\aligned
C_q v=(u_i,u'_i)_{i\in I},\textrm{where } \; u_i=v_i, \; 
u'_i=v_i-\sum_{h: i-j} v_{j}
\endaligned
\end{equation}
where $h$ runs over all arrows with an endpoint $i$.

\begin{definition}\label{def:l-dominant}
Given dimension vectors $w=(w_i,w'_i)_{i\in I}\in\mathbb{Z}_{\ge0}^{2n}$, $v=(v_i)_{i\in I}\in\mathbb{Z}_{\ge0}^n$. 
We say that
$(v,w)$ is $l$-\emph{dominant} if $w-C_qv\ge0$, or equivalently, if
\begin{equation}\label{eq:l-dominant}
\textrm{ For all $\forall i\in I$,}\quad  w_i\ge v_i ,\quad
w'_i \ge v_i-\sum_{h: i-j} v_j . 
\end{equation}
(Note that the condition \eqref{eq:l-dominant} implies \eqref{nonempty F}.)
Define
$$\D(w):=\{v\in\mathbb{Z}_{\ge0}^{n} \ | \ (v,w) \textrm{ is $l$-dominant} \}.$$
$$\D:=\{(v,w)\in\mathbb{Z}_{\ge0}^{n}\times\mathbb{Z}_{\ge0}^{2n} \ | \ (v,w) \textrm{ is $l$-dominant} \}.$$
\end{definition}

In the following, we define $\bar{v}$ with the property that $(\bar{v},w)$ is $l$-dominant and $\mathbf{E}_{v,w}=\mathbf{E}_{\bar{v},w}$ (which will be proved in Proposition \ref{prop:fiber}).

\begin{lemma}\label{lemma:vbar}
Let $w=(w_i,w_i')_{i\in I}\in\mathbb{Z}_{\ge0}^{2n}$, $v=(v_i)_{i\in I}\in\mathbb{Z}_{\ge0}^{n}$,  $V=\{v'\in\mathbb{Z}_{\ge0}^{n} \ | \ v'\le v \}$. 
Then $V\cap \D(w)$ has a unique maximal element $\bar{v}=(\bar{v}_i)_{i\in I}$ in the sense that every $v'\in V\cap \D(w)$ satisfies $v'\le \bar{v}$. 
More explicitly, 
\begin{equation}\label{eq:barv}
\bar{v}_i =\min\big(v_i,w_i,w'_i+\sum_{h:i-j} \min(v_j,w_j)\big).
\end{equation}
In particular: 

-- if $(v,w)$ satisfies \eqref{nonempty F}, then for $\beta\in I_1$ we have
$\bar{v}_\beta=\min(v_\beta,w_\beta)$;  

-- if $v\in\D(w)$, then $\bar{v}=v$. 
\end{lemma}
\begin{proof}
It suffices to prove \eqref{eq:barv}. It is obvious true if ``$=$'' is replaced by ``$\le$''. We are left to prove that $\bar{v}$ defined in \eqref{eq:barv} satisfies \eqref{eq:l-dominant}. By definition,  $w_i\ge \bar{v}_i$ for all $i\in I$. So we only need to show $w'_i \ge \bar{v}_i-\sum_{h:i-j} \bar{v}_j$ for all $i\in I$. 

We prove by contradiction. Assume for some $i\in I$, 
$w'_i < \bar{v}_i-\sum_{h:i-j} \bar{v}_j$. 
If $i$ is an isolated vertex (that is, not incident to any arrow in $\Omega$), then 
$w'_i\ge \min(v_i,w_i,w'_i) = \bar{v}_i = \bar{v}_i-\sum_{h:i-j} \bar{v}_j$, a contradiction. 
So we can assume that $i$ is adjacent to at least one vertex. 
Let $m_i=\min(v_i,w_i)$ for $i\in I$. Then $\bar{v}_i=\min(m_i,w'_i+\sum_{h:j-i}m_j)$,
and
\begin{equation}\label{eq:min mi}
\aligned
\min(m_i-w'_i,\sum_{h:i-j} m_j) 
&
=
\bar{v}_i-w'_i
>\sum_{h:i-j}\bar{v}_j
=\sum_{h:i-j} \min(m_j,w'_j+\sum_{h':j-k}m_k)\\
&\ge \sum_{h:i-j} \min(m_j,w'_j+ m_i),
\endaligned
\end{equation}
where the last inequality is because $\sum_{h':j-k}m_k\ge m_i$ (since $h$ is one of the possible $h'$).

If there exists an arrow $h:i-j_0$ such that $m_{j_0}\ge w'_{j_0}+m_i$, then 
$$m_i-w'_i
\ge \min_i (m_i-w'_i,\sum_{h:i-j} m_j)
\stackrel{\eqref{eq:min mi}}{>}
\sum_{h:i-j} \min(m_j,w'_j+ m_i)
\ge \min(m_{j_0},w'_{j_0}+m_i)
= w'_{j_0}+m_i$$
which gives a contradiction. 
So we can assume $\min(m_j,w'_j+ m_i)=m_j$ for all $j$ in the last $\sum$ of \eqref{eq:min mi}. Thus \eqref{eq:min mi} becomes
$$
\min(m_i-w'_i,\sum_{h:i-j} m_j) 
>\sum_{h: i-j} m_j
$$
which again gives a contradiction. 
\end{proof}

\begin{lemma}\label{lem:ab}
Fix $w=(w_i,w_i')_{i\in I}\in\mathbb{Z}_{\ge0}^{2n}$, $v=(v_i)_{i\in I}\in\mathbb{Z}_{\ge0}^{n}$. Let
$\pi: \tilde{\mathcal{F}}_{v,w}\to \mathbf{E}_{v,w}$ be the map defined in \eqref{pi}.

{\rm(a)}  For each $v'\le v$,  the variety $\myE^\circ_{v',w}$ defined in \eqref{Ust} is nonempty if and only if  $(v',w)$ is $l$-dominant. In this  case, it is nonsingular and locally closed in $\myE_w$ (so is also locally closed in $\myE_{v,w}$), and is irreducible and rational. 
Thus the variety $\mathbf{E}_{v,w}$ has a stratification 
$$\mathbf{E}_{v,w}=\bigcup_{v'} \myE^\circ_{v',w}\, ,\quad
\textrm{where $v'\le v$ and $(v',w)$ are $l$-dominant}.$$ 
 
{\rm(b)} For each $v'\le v$ with $(v',w)$ being $l$-dominant, the restricted projection $\pi^{-1}(\myE^\circ_{v',w})\to \myE^\circ_{v',w}$ is a Zariski locally trivial $\mathcal{M}$-bundle, 
where $\mathcal{M}$ itself is a   
$$\Big(\prod_{\beta\in I_1} Gr(v_\beta-v'_\beta,w_\beta' -v'_\beta +\sum_{h:s(h)=\beta} v_{t(h)} \Big) \textrm{-bundle over }\prod_{\alpha\in I_0} Gr(v_\alpha-v'_\alpha,w_\alpha-v'_\alpha)$$ 
defined as
$$\aligned
\mathcal{M}:=\Big\{(X'_i)_{i\in I}  \Big|
& X_\alpha'\in Gr(v_\alpha-v_\alpha',w_\alpha-v_\alpha') \ (\forall \alpha\in I_0), \\
&X_\beta'\in Gr\big(v_\beta-v_\beta',\mathbb{A}^{w_\beta' -v_\beta' +\sum_{h:s(h)=\beta}v'_{t(h)} }\oplus\bigoplus_{h:s(h)=\beta}(X_{t(h)}' \big) \ (\forall \beta\in I_1)\Big\}.
\endaligned$$
In particular, $\mathcal{F}_{v,w}$ is a 
$\prod_{\beta\in I_1}Gr(v_\beta,w_\beta'+\sum_{h:s(h)=\beta} v_{t(h)})$-bundle over $\prod_{\alpha\in I_0} Gr(v_\alpha,w_\alpha)$. 
\end{lemma}
\begin{proof} The proof is similar to the proof of \cite[Lemma 4.6]{L}.

(a) $\myE^\circ_{v',w}\subseteq \myE_w$ is locally closed because 
$$\aligned
\myE^\circ_{v',w} 
=&\{ (x_i,y_h)_{i\in I, h\in \Omega} \in\mathbf{E}_w\, |\, 
\mathrm{rank}A_i\le v'_i (\forall i) \}\\
&-\bigcup_{j\in I} \{ (x_i,y_h)_{i\in I, h\in \Omega} \in\mathbf{E}_w\, |\, 
\mathrm{rank}A_j\le v'_j-1,\;  \mathrm{rank}A_i\le v'_i (\forall i) \}
\endaligned
$$

To show $\myE^\circ_{v',w}$ is nonsingular, we consider the following open covering
\begin{equation}\label{UIJ}
\myE^\circ_{v',w}=\bigcup U_{\mathbf{J}}
\end{equation}
where ${\bf J}=(J_i)_{i\in I}$, $J_i\subseteq\{1,\dots,w_i\}$ and $|J_i|=v_i'$, $U_{\bf J}\subseteq \myE^\circ_{v',w}$ consists of those $(x_i,y_h)$ such that for all $i\in I_0$ (resp.~$i\in I_1$), the rows (resp.~column) of $A_i$ indexed by $J_i$ are linearly independent.
We will show that 
$U_{\bf J}$ is isomorphic to an open subset $\tilde{\mathbb{A}}^\circ$ of an affine space $\tilde{\mathbb{A}}$,
 which implies that $\myE^\circ_{v',w}$ is nonsingular. 

Without loss of generality, assume $J_i=\{1,\dots,v_i'\}$. 
Denote 
$$
\aligned
\tilde{\mathbb{A}}
&:=\prod_{\alpha\in I_0}\mathbb{A}^{(w_\alpha-v_\alpha')\times v_\alpha'}
\times 
\prod_{\beta\in I_1}\mathbb{A}^{v_\beta'\times (w_\beta-v_\beta')}
\times
\prod_{\alpha\in I_0}\mathbb{A}^{v_\alpha'\times w_\alpha'}\times
\prod_{\beta\in I_1} \mathbb{A}^{w_\beta'\times v_\beta'}\times 
\prod_{h\in\Omega} \mathbb{A}^{v'_{t(h)}\times v'_{s(h)}}\\
&\cong\mathbb{A}^{\sum_{i\in I}v'_iw_i+\sum_{h\in\Omega} v'_{t(h)}v'_{s(h)}}
\endaligned
$$ 
and write its elements as $(P_i,Z_i,C_h)_{i\in I, h\in \Omega}$ where the size of each matrix is given in the following table: (where $\alpha\in I_0$, $\beta\in I_1$, $h\in\Omega$)
\begin{center}
 \begin{tabular}{|c|c|c|c|c|} 
 \hline
 $P_\alpha$ & $P_\beta$ & $Z_\alpha$ & $Z_\beta$ & $C_h$ \\
 \hline
 $(w_\alpha-v'_\alpha)\times v'_\alpha$ & $v_\beta'\times (w_\beta-v_\beta')$ & $v_\alpha'\times w_\alpha'$ & $w_\beta'\times v_\beta'$ & 
$v'_{t(h)}\times v'_{s(h)}$\\
 \hline
\end{tabular}
\end{center}

Then define a morphism
$$\aligned
\varphi: \tilde{\mathbb{A}} &\to \myE_w\\
(P_i,Z_i,C_h)_{i\in I, h\in\Omega}&\mapsto 
(x_i,y_h)_{i\in I, h\in\Omega}
\endaligned
$$
where, 
$x_\alpha=\begin{bmatrix}Z_\alpha\\P_\alpha  Z_\alpha\end{bmatrix}$ for $\alpha\in I_0$;
$x_\beta=\begin{bmatrix}Z_\beta&Z_\beta P_\beta\end{bmatrix}$ for $\beta\in I_1$; 
$y_h=\begin{bmatrix}C_h&C_hP_\beta\\ P_\alpha C_h&P_\alpha C_hQ_\beta \end{bmatrix}
$ for $h:\beta\to\alpha \in \Omega$. 
It is easy to see that ${\rm im}\varphi\subseteq\myE_{v',w}$.

For $\alpha\in I_0$, $\beta\in I_1$, let $r_1,\dots,r_s$ and $r'_1,\dots,r'_t$ be given as in Definition \ref{def:Ai}. Define 
\begin{equation}\label{eq:A'}
A'_\alpha:=[Z_\alpha\ C_{r_1},\  \cdots,\ C_{r_s}]_\text{hor},\quad  A'_\beta:=[Z_\beta\ C_{r'_1},\  \cdots,\ C_{r'_t}]_\text{vert}
\end{equation}
Note that the size of $A'_\alpha$ is $v'_\alpha\times(w'_\alpha+\sum_{h:\ t(h)=\alpha} v'_{s(h)})$, so its number of rows $\le$ its number of columns when $(v',w)$ is $l$-dominant;  
similarly, the size of $A'_\beta$ is
$(w'_\beta+\sum_{h:\ s(h)=\beta} v'_{t(h)}) \times v'_\beta$, 
so its number of columns $\le$ its number of rows when $(v',w)$ is $l$-dominant.

Define 
$
\tilde{\mathbb{A}}^\circ := \big\{(P_i, Z_i, C_h)\in \tilde{\mathbb{A}} \; \Bigg|\;  
\rank A'_i=v'_i\; (\forall i\in I)\big\}
$
to be an open subset of $\tilde{\mathbb{A}}$.

Similar to the proof of \cite[Lemma 4.6]{L}, it can be shown that: 

(i) $\varphi$ restricts to an isomorphism 
$\varphi |_{\tilde{\mathbb{A}}^\circ}: \ \tilde{\mathbb{A}}^\circ \stackrel{\cong}{\to} U_{\bf J}$; 
\smallskip

(ii) $\myE^\circ_{v',w}$ is nonempty if and only if  $(v',w)$ is $l$-dominant; 
\smallskip

(iii) if $\myE^\circ_{v',w}$ is nonempty, it is irreducible.
\smallskip

It follows that $\myE^\circ_{v',w}$ is rational. 
\smallskip

As pointed out in the proof of \cite[Lemma 4.6]{L}, the proof of the ``if'' part of (ii) and the proof of (iii) therein make use of some generic choice. Alternatively, We give the new proofs below that do not use the generic choice.

For the ``if'' part of (ii): it suffices to show that $\tilde{\mathbb{A}}^\circ$ is nonempty by constructing a point on it. Note that $P_i$ are irrelevant so we let them be zero matrices. We construct $Z_i,C_h$ as follows: start with all $Z_i$, $C_h$ being zero matrices, then for each $A'_\alpha$ and $A'_\beta$ defined in \eqref{eq:A'}, change the diagonal entries (that is, the $(j,j)$- entry for all $j$) to 1 and make the corresponding change in $Z_i$ and $C_h$. We showed that the resulting $(P_i, Z_i, C_h)$ satisfies the equality  $\rank A'_i=v'_i\; (\forall i\in I)$. Without loss of generality, we only need to show $\rank A'_\alpha = v'_\alpha$ for  $\alpha\in I_0$. We claim the $A'_\alpha$ is an ``upper triangular matrix'' in the sense that its $(i,j)$-entries are $0$ whenever $i>j$; then together with the fact that its diagonal entries are 1, we can conclude that $\rank A'_\alpha=v'_\alpha$ (that is, $A'_\alpha$ has full row rank). Assume the contrary that some $(i,j)$-entry of $A'_\alpha$ is $1$ with $i>j$. Then it corresponds to the $(i,j')$-entry of $C_h$ for some arrow $h:\beta\to\alpha$ and $j'\le j$; moreover it must correspond to a diagonal entry of $A'_\beta$. So it is simultaneously the $(j',j')$-entry of $A'_\beta$ and the $(i,j')$-entry of $A'_\alpha$. In particular, $i=j'$. However, $i>j\ge j'$, a contradiction. 

For (iii), assuming $(v',w)$ is $l$-dominant (so $\myE^\circ_{v',w}$ is nonempty), we shall prove that $\myE^\circ_{v',w}$ is irreducible. It suffices to show that $U_{\bf J}\cap U_{\bf J'}\neq \emptyset$ where ${\bf J}=(J_i)_{i\in I}$ and ${\bf J'}=(J'_i)_{i\in I}$ satisfy the following condition: there exists $i_0\in I$ such that $J_{i_0}$ and $J'_{i_0}$ only differ by one element, and $J_i=J'_i$ for $i\neq i_0$. Without loss of generality, we can assume $i_0=\alpha\in I_0$, $J_\alpha=\{1,\dots,v'_\alpha\}$, $J'_\alpha=\{1,\dots,v'_{\alpha-1},v'_{\alpha}\}$. Define $Z_i (i\in I), C_h (h\in\Omega)$ to be the same as the last paragraph, $P_i (i\in I)$ be any matrices, under the condition that the first row of $P_\alpha$ is $[0,\dots,0,1]$. Then in the $v'_\alpha$-th row and the $(v'_\alpha+1)$-th row of $A'_\alpha$ are identical. It is then obvious that $\varphi((P_i,Z_i,C_h))$ is a point in $U_{\bf J}\cap U_{\bf J'}$, thus the latter is nonempty.  
\medskip

(b) Similar to the proof of \cite[Lemma 4.6]{L}: let 
$$\aligned
&\mathcal{X} := \big\{(X_i, Y_i)_{i\in I} \ | \ (Y_i)_{i\in I} \in \mathcal{F}_{v',w}, (X_i)_{i\in I}\in\mathcal{F}_{v,w}, Y_i\subseteq X_i (\forall i\in I) \big\}, \\
& f: (x_i,y_h)_{i\in I, h\in\Omega}\mapsto (X_i)_{i\in I},\;  \textrm{ where } X_i=\text{im}A_i ,\\
& q: (X_i,Y_i)_{i\in I} \mapsto (Y_i)_{i\in I} .
\endaligned
$$ 
Then  $\pi^{-1}(\myE^\circ_{v',w})\to \myE^\circ_{v',w}$ is the pullback of the locally trivial bundle $q:\mathcal{X}\to \mathcal{F}_{v',w}$:
$$\xymatrix{\pi^{-1}(\myE^\circ_{v',w}) \ar[d]^\pi\ar[r]& \mathcal{X}\ar[d]^q\\ \myE^\circ_{v',w}\ar[r]^f & \mathcal{F}_{v',w}}$$
and the fibers of $q$ are isomorphic to $\mathcal{M}$.  
\end{proof}

\begin{proposition}\label{prop:fiber}
Let  $v, w, \pi$ be the same as Lemma \ref{lem:ab}
and $\bar{v}$ be as in Lemma \ref{lemma:vbar}. 

{\rm(a)} $\mathbf{E}_{v,w}=\mathbf{E}_{\bar{v},w}$ is irreducible. Moreover, $\mathbf{E}^\circ_{\bar{v},w}$ is its largest stratum; so  $\mathbf{E}_{v,w}=\overline{\mathbf{E}^\circ_{\bar{v},w}}$, the Zariski closure of $\mathbf{E}^\circ_{\bar{v},w}$.

{\rm(b)} Assume \eqref{nonempty F}. Then  $\tilde{\mathcal{F}}_{v,w}$ and $\mathcal{F}_{v,w}$ are irreducible and nonsingular and $\pi$ is surjective. Moreover,
\begin{flalign*}
\dim \mathcal{F}_{v,w}
=d_{v,w}
&:=\sum_{\beta\in I_1} ({w_\beta'+\sum_{h:s(h)=\beta} v_{t(h)}}-v_\beta){v_\beta}
+\sum_{\alpha\in I_0} ({w_\alpha}-v_\alpha){v_\alpha} \\ 
& =\sum_{\alpha\in I_0}w_\alpha v_\alpha +\sum_{\beta\in I_1} w'_\beta v_\beta -\sum_{i\in I}v_i^2+\sum_{h\in\Omega} v_{s(h)}v_{t(h)}
\\
\dim \tilde{\mathcal{F}}_{v,w}=\tilde{d}_{v,w}
&:=\dim \mathcal{F}_{v,w}+\sum_{\alpha\in I_0} v_\alpha w'_\alpha+\sum_{\beta\in I_1}v_\beta w_\beta\\
&=\sum_{\alpha\in I_0}(w_\alpha+w'_\alpha) v_\alpha +\sum_{\beta\in I_1} (w_\beta+w'_\beta) v_\beta -\sum_{i\in I}v_i^2+\sum_{h\in\Omega} v_{s(h)}v_{t(h)}\\
\dim \mathbf{E}_{v,w}
&=\dim \mathbf{E}_{\bar{v},w}
=\tilde{d}_{\bar{v},w}\\  
\end{flalign*}

{\rm(c)} Further assume that $(v,w)$ is $l$-dominant. Then $\pi$ is birational. In this case, it restricts to an isomorphism $\pi^{-1}({\bf E}^\circ_{v,w})\stackrel{\cong}{\longrightarrow} {\bf E}^\circ_{v,w}$. 

\end{proposition}

\begin{proof}
Similar to the proof of \cite[Proposition 4.8]{L}.
\end{proof}

\section{Decomposition theorem}

The following fact is proved by Nakajima in \cite[Theorem 14.3.2]{Nakajima_JAMS}  using the representation theory of quantum affine algebras. We will give a proof similar to  \cite[Theorem 5.1]{L}; instead of representation theory, we use the algebraic version of the Transversal Slice Theorem.

\begin{theorem}\label{thm:trivial local system}
Assume ${\mytildeF_{v,w}}\neq\emptyset$. The local system appeared in the BBDG decomposition for  $\pi_*IC_{\mytildeF_{v,w}}$ are all trivial. Thus,
\begin{equation}\label{piICtildeF}
\pi_*(IC_{\mytildeF_{v,w}})=\bigoplus_{v'}\bigoplus_{d}a^d_{v,v';w}IC_{\mathbf{E}_{v',w}}[d]
\end{equation}
where $v'\le v$ satisfies the condition that $(v',w)$ is $l$-dominant, and $a^d_{v,v';w}\in\mathbb{Z}_{\ge0}$. 
\end{theorem}

The proof is delayed to \S\ref{subsection:trivial local system}. First, recall: 
\begin{theorem}[BBDG Decomposition Theorem]\cite{BBD}\label{BBDthm}
Let $f:X\to Y$ be a proper algebraic morphism between complex algebraic varieties. 
Then there is a finite list of triples $(Y_a,L_a,n_a)$, where for each $a$, $Y_a$ is locally closed smooth irreducible algebraic subvariety of $Y$, $L_a$ is a semisimple local system on $Y_a$, $n_a$ is an integer,  such that:
\begin{equation}\label{eq:decomp thm}
f_* IC_X \cong \bigoplus_{a} IC_{\overline{Y_a}}(L_a)[n_a]
\end{equation}
Moreover, even though the isomorphism is not necessarily canonical, the direct summands appeared on the right hand side are canonical. 
\end{theorem}

\subsection{Algebraic Transversal Slice Theorem}
Nakajima proves an analytic transversal slice theorem \cite[\S3]{Nakajima_JAMS} and we strengthen the result to an algebraic version for the varieties studied in this paper using an explicit construction similar to the one in \cite{L}. 

\begin{lemma}[Algebraic Transversal Slice Theorem]\label{lem:transversal slice}
Let $p\in \myE_{v,w}$ be a point in the stratum $\myE^\circ_{v^0,w}$. So $v^0\le v$. Define 
$$w^\perp=w-C_qv^0=(w^\perp_i, {w^\perp_i}')_{i \in I},
\quad
v^\perp=v-v^0=(v^\perp_i)_{i\in I},$$ 
that is, 
$w^\perp_i=w_i-v_i^0$,
${w^\perp_i}'=w'_i-v_i^0+\sum_{h:j-i} v_j^0$, 
$v^\perp_i=v_i-v_i^0$.
Take $U^\perp=\myE_{v^\perp,w^\perp}$.
Then there exist Zariski open neighborhoods $U\subseteq\myE_{v,w}$ of $p$, $U^0\subseteq \myE^\circ_{v^0,w}$ of $p$, 
 and isomorphisms $\varphi,\psi$ making the following diagram commute: 
$$\xymatrix{\mytildeF_{v,w} \ar@{}[r]|{\supseteq} &\pi^{-1}U\ar[d]^\pi\ar[r]^-\varphi_-\cong&U^0\times\pi^{-1}(U^\perp)\ar[d]^{1\times\pi}\ar@{}[r]|{\subseteq}&\myE^\circ_{v^0,w}\times \mytildeF_{v^\perp,w^\perp}\\
\mathbf{E}_{v,w}\ar@{}[r]|{\supseteq}&U\ar[r]^-\psi_-\cong & U^0\times U^\perp\ar@{}[r]|{\subseteq} & \myE^\circ_{v^0,w}\times \mathbf{E}_{v^\perp,w^\perp}}$$
We can further assume $U$ and $U^0$ are both invariant under the natural $\mathbb{C}^*$-action (by multiplying all entries with the same scalar). 

Moreover, $\psi(p)=(p,0)$, and the diagram is compatible with the stratifications in the sense that 
$\psi({\bf E}^\circ_{u,w}\cap U)=U_0\times {\bf E}^\circ_{u^\perp,w^\perp}$ 
for each $u$ satisfying $v^0\le  u\le v$, where $u^\perp=u-v^0$. 
\end{lemma}
\begin{proof}
(1) We will define $\psi$. 
%
Recall that the point $p$ determines $n$ matrices as in \eqref{AB}, denoted $A_i(p)$ ($i\in I$). Since $p\in \myE^\circ_{v^0,w}$, without loss of generality, we can assume that, 
for $i\in I_0$, the first $v_i^0$ row vectors of $A_i(p)$  are linearly independent;
and for  $i \in I_1$, the first $v_i^0$ column vectors of $A_i(p)$ are linearly independent.   
Then for $\alpha\in I_0$, $\beta\in I_1$, there exist matrices with the indicated sizes:
\begin{center}
 \begin{tabular}{|c|c|c|c|c|} 
 \hline
 $P_\alpha(p)$ & $P_\beta(p)$ & $Z_\alpha(p)$ & $Z_\beta(p)$ & $C_h(p)$ \\
 \hline
 $(w_\alpha-v^0_\alpha)\times v^0_\alpha$ & $v^0_\beta\times (w_\beta-v^0_\beta)$ & $v^0_\alpha\times w_\alpha'$ & $w_\beta'\times v^0_\beta$ & 
$v^0_{t(h)}\times v^0_{s(h)}$\\
 \hline
\end{tabular}
\end{center}
such that (where for simplicity we drop the parameter $p$ for $Z_\alpha$, $C_{r_1}$, etc.)
\begin{equation*}
\aligned
&A_\alpha(p)=\begin{bmatrix}Z_\alpha&C_{r_1}&C_{r_1}P_{\beta_{r_1}}&\cdots&C_{r_s}&C_{r_s}P_{\beta_{r_s}} \\ P_\alpha Z_\alpha &P_\alpha C_{r_1}&P_\alpha C_{r_1} P_{\beta_{r_1}}&\cdots&P_\alpha C_{r_s} &P_\alpha C_{r_s} P_{\beta_{r_s}}
\end{bmatrix} ,
 \\
 &A_\beta(p)
=\begin{bmatrix} 
Z_\beta & Z_\beta P_\beta \\ C_{r'_1}&C_{r'_1}P_\beta \\ P_{\alpha_{r'_1}}C_{r'_1}& P_{\alpha_{r'_1}}C_{r'_1}  P_\beta \\ \vdots &\vdots\\ C_{r'_t}& C_{r'_t} P_\beta \\ P_{\alpha_{r'_t}}C_{r'_t} & P_{\alpha_{r'_t}}C_{r'_t} P_\beta
\end{bmatrix} , 
\endaligned
\end{equation*}
and that the corresponding $A'_\alpha(p)$ of size $v^0_\alpha\times(w'_\alpha+\sum_{h:\ t(h)=\alpha} v^0_{s(h)})$ and $A'_\beta(p)$ of size $(w'_\beta+\sum_{h:\ s(h)=\beta} v^0_{t(h)}) \times v^0_\beta$, as defined in \eqref{eq:A'}, are full rank. 

\medskip
Step 1: define $U$ and $U^0$.  
For $\alpha\in I_0$, define $M_\alpha$ to be the index set of the columns of $A_\alpha(p)$ that correspond to the columns of $A'_\alpha(p)$. 
The submatrix $A'_\alpha(p) := A_\alpha(p)_{[1,\dots,v_\alpha^0;M_\alpha]}$ has rank $v_\alpha^0$, thus there exists a subset  $J_\alpha\subseteq M_\alpha$ of cardinality $v_\alpha^0$ such that $A_\alpha(p)_{[1,\dots,v_\alpha^0;J_\alpha]}$ is an invertible square matrix. 
Similarly, for $\beta\in I_1$, define $M_\beta$ to be the index set of the rows of $A_\beta(p)$ that correspond to the rows of $A'_\beta(p)$, and there exists $J_\beta\subseteq M_\beta$ of cardinality $v_\beta^0$ such that $A_\beta(p)_{[J_\beta; 1,\dots,v_\beta^0]}$ is an invertible square matrix. 
 Define an open subset $U$ (which depends on $J_\alpha$) of  $\myE_{v,w}$ as
 $$U := \big\{ q\in \myE_{v,w}\ \big| \ \textrm{ $A_\alpha(q)_{[1,\dots,v_\alpha^0;J_\alpha]}$ ($\forall \alpha\in I_0$) and  $A_\beta(q)_{[J_\beta; 1,\dots,v_\beta^0]}$ ($\forall \beta\in I_1$) are invertible} \big\}.$$ 
 Define $U^0 := U\cap  \myE_{v^0,w}$ which is an open subset of  $\myE^\circ_{v^0,w}$.
 So
 $$U^0=\big\{ q\in \myE_{v^0,w}\ \big| \ \textrm{ $A_\alpha(q)_{[1,\dots,v_\alpha^0;J_\alpha]}$ ($\forall \alpha\in I_0$) and  $A_\beta(q)_{[J_\beta; 1,\dots,v_\beta^0]}$ ($\forall \beta\in I_1$) are invertible} \big\}.$$ 
Obviously, both $U$ and $U^0$ are invariant under the natural $\mathbb{C}^*$-action. 
\medskip
 
Step 2: in the rest we assume $q\in {\bf E}^\circ_{u,w}\cap U$.  We construct $q^0$ as follows. 
Denote the following (where we drop the parameter $q$ of $Z_\alpha$, $Z'_\alpha$, etc.)
$$A_\alpha(q)=\begin{bmatrix}Z_\alpha&C_{r_1}&C_{r_1}'&\cdots&C_{r_s}&C_{r_s}'\\ Z'_\alpha&C_{r_1}''&C_{r_1}''' &\cdots&C_{r_s}''&C_{r_s}'''
\end{bmatrix},
\quad
A_\beta(q)
=\begin{bmatrix}Z_\beta& Z'_\beta\\ C_{r'_1}&C_{r'_1}' \\ C_{r'_1}''& C_{r'_1}''' \\ \vdots &\vdots\\ C_{r'_t}& C_{r'_t}' \\ C_{r'_t}'' & C_{r'_t}'''
\end{bmatrix} 
$$
where the sizes of the blocks are as follows: 
\begin{center}
 \begin{tabular}{|c|c|c|c|c|} 
 \hline
block & $Z_\alpha(q)$ & $Z_\alpha'(q)$ &$Z_\beta(q)$ &$Z_\beta'(q)$ \\
 \hline
size   & $v_\alpha^0\times w'_\alpha$ & $(w_\alpha-v_\alpha^0)\times w'_\alpha$ 
&$w'_\beta\times v_\beta^0$
&$w'_\beta\times(w_\beta-v_\beta^0)$
\\
 \hline
\end{tabular}
\vspace{10pt}

 \begin{tabular}{|c|c|c|c|c|} 
 \hline
block &  $C_j(q)$ &$C'_j(q)$ & $C''_j(q)$ & $C'''_j(q)$\\
 \hline
size 
& $v_{\alpha_j}^0\times v_{\beta_j}^0$ 
& $v_{\alpha_j}^0\times (w_{\beta_j}-v_{\beta_j}^0)$ 
& $(w'_{\alpha_j}-v_{\alpha_j}^0)\times v_{\beta_j}^0 $ 
&$(w'_{\alpha_j}-v_{\alpha_j}^0) \times (w_{\beta_j}-v_{\beta_j}^0)$\\
 \hline
\end{tabular}
\end{center}

Since $A_\alpha(q)_{[1,\dots,v_\alpha^0;J_\alpha]}$ ($\forall \alpha\in I_0$) and  $A_\beta(q)_{[J_\beta; 1,\dots,v_\beta^0]}$ ($\forall \beta\in I_1$) are invertible, there are unique $(w_\alpha-v_\alpha^0)\times v_\alpha^0$ -matrix $P_\alpha(q)$ ($\forall \alpha\in I_0$) and $v_\beta^0\times (w_\beta-v_\beta^0)$ -matrix $P_\beta(q)$  ($\forall \beta\in I_1$) such 
that (where we omit the parameter $q$ of $P_\alpha$, $P_\beta$)
\begin{equation}\label{AJBI}
A_\alpha(q)_{[-;J_\alpha]}
=\begin{bmatrix}
A_\alpha(q)_{[1,\dots,v_\alpha^0;J_\alpha]}\\P_\alpha A_\alpha(q)_{[1,\dots,v_\alpha^0;J_\alpha]}\end{bmatrix}, \ 
A_\beta(q)_{[J_\beta;-]}=\bigg[ A_\beta(q)_{[J_\beta;1,\dots,v_\beta^0]}\;\; A_\beta(q)_{[J_\beta;1,\dots,v_\beta^0]}P_\beta\bigg].
\end{equation}
Define $q^0=(x_i^0,y_h^0)_{i\in I,h\in\Omega}$  where
$$x_\alpha^0=\begin{bmatrix}Z_\alpha \\P_\alpha Z_\alpha\end{bmatrix}, 
x_\beta^0=\begin{bmatrix}Z_\beta&Z_\beta P_\beta\end{bmatrix},
y_{j}^0=\begin{bmatrix}C_{j}&C_j P_{\beta_j} \\ P_{\alpha_j} C_j &P_{\alpha_j}C_j P_{\beta_j}\end{bmatrix}$$
(where we omitted the parameter $q$). Then $q^0$ is in $U^0$. %

Note that the projection $U\to U^0$, $q\mapsto q^0$ is compatible with the $\mathbb{C}^*$-action. 
\medskip

Step 3:   we let $U^\perp=\myE^\circ_{v^\perp,w^\perp}$ and construct $q^\perp=(x_i^\perp,y_h^\perp)_{i\in I, h\in\Omega}$ as follows. Define  
$$
\widetilde{P}_\alpha := \begin{bmatrix}{\bf I}_{v_\alpha^0}&0\\ -P_\alpha&{\bf I}_{w_\alpha-v_\alpha^0}\end{bmatrix} (\forall \alpha\in I_0),\quad  
\widetilde{P}_\beta := \begin{bmatrix}{\bf I}_{v_\beta^0}&-P_\beta\\ 0&{\bf I}_{w_\beta-v_\beta^0}\end{bmatrix} (\forall \beta\in I_1) 
$$
and for $\alpha\in I_0$, define (where we omit the parameter $q$)
$$\widetilde{A}_\alpha(q) := [\widetilde{P}_\alpha x_\alpha,  \widetilde{P}_\alpha y_{r_1}\widetilde{P}_{\beta_{r_1}}, \cdots, \widetilde{P}_\alpha y_{r_s}\widetilde{P}_{\beta_{r_s}}]_{\rm hor},
$$
It is easy to see that $\rank \widetilde{A}_\alpha(q)=\rank A_\alpha(q)=u_\alpha$, and $\widetilde{A}_\alpha(q)_{[1,\dots,v_\alpha^0; J_\alpha]}= \begin{bmatrix}A_\alpha(q)_{[1,\dots,v_\alpha^0;J_\alpha]}\\0\end{bmatrix}$.

Note that the space spanned by the top $v_\alpha^0$ rows of $\widetilde{A}_\alpha(q)$ intersects only at the origin with the space spanned by the rest rows.
Since  $\rank\widetilde{A}_\alpha(q)=\rank A_\alpha(q)=u_\alpha \le v_\alpha$, and that 
$\rank \widetilde{A}_\alpha(q)_{[1,\dots,v_\alpha^0;-]}=v_\alpha^0$ (because 
$\widetilde{A}_\alpha(q)_{[1,\dots,v_\alpha^0;J_\alpha]}
=A_\alpha(q)_{[1,\dots,v_\alpha^0;J_\alpha]}$ is invertible),
we have
$$\rank \widetilde{A}_\alpha(q)_{[v_\alpha^0+1,\dots,w_\alpha; -]} = \rank \widetilde{A}_\alpha(q)-\rank \widetilde{A}_\alpha(q)_{[1,\dots,v_\alpha^0;-]}
=u_\alpha-v_\alpha^0=u_\alpha^\perp\le v_\alpha^\perp.$$
We want $A_\alpha(q)^\perp$ to be 
$\widetilde{A}_\alpha(q)_{[v_\alpha^0+1,\dots,w_\alpha; \{1,\dots,w_\alpha'+\sum_{h:\;  t(h)=\alpha } w_{s(h)}\}\setminus J_\alpha]}$ with an appropriate rearrangement of the columns; to be more precise (recall that $M_\alpha$ is defined in Step 1):
$$A_\alpha(q)^\perp=\Big[  \widetilde{A}_\alpha(q)_{[v_\alpha^0+1,\dots,w_\alpha; M_\alpha\setminus J_\alpha]} \ \vline \  P_\alpha C_{r_1}P_{\beta_{r_1}}-P_\alpha C'_{r_1}-C''_{r_1}P_{\beta_{r_1}}+C_{r_1}''' \ \vline \  \cdots \Big]$$
For this, we denote $J'_\alpha=\iota^{-1}(J_\alpha)$  where $\iota:\{1,\dots, |M_\alpha|\}\to M_\alpha$ is the order preserving bijection, and define 
$$\aligned
&x_\alpha^\perp := \text{the matrix obtained from $[-P_\alpha Z_\alpha+Z'_\alpha, -P_\alpha C_{r_1}+C_{r_1}'', \cdots, -P_\alpha C_{r_s}+C_{r_s}'']_{\rm hor}$ }\\
&\hspace{30pt} \text{by deleting  columns of indices in $J'_\alpha$ (which are zero columns)},\\ 
&y_{h}^\perp := P_\alpha C_h P_\beta -P_\alpha C_h'-C_h''P_\beta+C_h''' \quad \textrm{ for arrow } h:\beta\to \alpha. 
\endaligned
$$
Note that the column spaces of $A_\alpha(q)^\perp$ and $\widetilde{A}_\alpha(q)_{[v_\alpha^0+1,\dots,w_\alpha; -]}$ are the same, which implies
\begin{equation}\label{eq:rank Aqperp}
\rank A_\alpha(q)^\perp= u_\alpha^\perp
\end{equation}

Similarly, we define 
$x_\beta^\perp$ to be obtained from 
$[-Z_\beta P_\beta+Z'_\beta, -C_{r'_1}P_\beta+C'_{r'_1}, \cdots, -C_{r'_t}P_\beta+C'_{r'_t}]_{\rm vert}$ by deleting rows of indices in $J_\beta$ 
where  $J'_\beta=\iota^{-1}(J_\beta)$ (here $\iota:\{1,\dots, |M_\beta|\}\to M_\beta$ is the order preserving bijection). 
We have
\begin{equation}\label{eq:rank Bqperp}
\rank A_\beta(q)^\perp= u_\beta^\perp
\end{equation}

This complete the construction of $q^\perp$. 
Note that the projection $U\to U^\perp$, $q\mapsto q^\perp$ is also compatible with the $\mathbb{C}^*$-action.

Note that $\psi$ is a morphism because it can be expressed in terms of matrix additions, multiplications, and inverses, so can be expressed as rational functions. 

\medskip
\noindent (2) 
We show that $\psi$ is an isomorphism by constructing its inverse $\psi^{-1}$. Assume $q^0,q^\perp$ are given. 
From $q^0$ we can recover $Z_i, P_i,C_h$ by rational functions. 
From $q^\perp$ we can recover $Z_i', C_h', C_h'', C_h'''$ by rational functions. This uniquely determines $q$. It is routine to check that $q$ is well-defined,  that $\psi^{-1}$ is a morphism and that $\psi^{-1}$ is indeed the two-sided inverse of $\psi$. 

\medskip
\noindent (3) 
We define $\varphi$. Given $(q=(x_i,y_h),X_i)\in \pi^{-1}U$, define 
$\varphi(q,X_i) := (q^0,(q^\perp,X_i^\perp))$, where $q^0$ and $q^\perp$ are defined in (1), and $X_i^\perp$ are defined as follows:

We define $X_\alpha^\perp$  ($\alpha\in I_0$) as follows. 
For ${\bf u}\in X_\alpha$, define
$$\rho^0({\bf u}) := \begin{bmatrix}{\bf I} &0\\ P_\alpha &0\end{bmatrix} {\bf u},\quad
\rho^\perp({\bf u}) := {\bf u}-\rho^0({\bf u})= \begin{bmatrix}0 &0\\ -P_\alpha&{\bf I}\end{bmatrix} {\bf u}.$$ 
Then the top $v_\alpha^0$ entries of $\rho^\perp({\bf u})$ are 0. Define
$$X_\alpha^0 := \{\rho^0({\bf u}) \ | \ {\bf u}\in X_\alpha \}, \quad 
X_\alpha^\perp := \{\rho^\perp({\bf u}) \ | \ {\bf u}\in X_\alpha \}
=\Big\{\begin{bmatrix} 0\\-P_\alpha {\bf u}_1+{\bf u}_2\end{bmatrix}\; \Big| \; \begin{bmatrix}{\bf u}_1 \\{\bf u}_2 \end{bmatrix}\in X_\alpha \Big\}$$
Then $\dim X_\alpha^0=v_\alpha^0$, 
$\dim X_\alpha^\perp
= v_\alpha^\perp$. Note that $X_\alpha^\perp=X_\alpha\cap (0\oplus \mathbb{A}^{w_\alpha-v_\alpha^0})$ is the set of vectors in $X_\alpha$ with the first $v_\alpha^0$ entries being 0. 
\smallskip

We define $X_\beta^\perp$  ($\beta\in I_1$)  as follows. 
Let $r'_1,\dots,r'_t$ be defined as in Definition \ref{def:Ai}. 
Define 
$$
P^*_\beta := \begin{bmatrix}
{\bf I}&0&\cdots&0\\0&
\widetilde{P}_{\alpha_{r'_1}}&\cdots& 0\\ \vdots&\vdots&\ddots&\vdots\\0&0&\cdots&\widetilde{P}_{\alpha_{r'_t}}
\end{bmatrix},\quad 
\widetilde{X}_\beta := P^*_\beta X_\beta.$$
Denote $\gamma|_{J_\beta}\in\mathbb{A}^{v_\beta^0}$ to be the vector obtained from $\gamma$ by only keeping its  entries of indices in $J_\beta$, where $J_\beta$ is defined in (1) Step 1. Define 
$$
X_\beta^\perp :=
\big\{\gamma 
\in \widetilde{X}_\beta
\ {\textrm {such that} } \
\gamma|_{J_\beta}={\bf 0}\big\}. 
$$

We now show that $\dim X_\beta^\perp=v_\beta^\perp$. Note that $\dim \widetilde{X}_\beta=\dim X_\beta=v_\beta$. 
Note that 
$$
\{ \gamma|_{J_\beta} \textrm{ for } \gamma\in \widetilde{X}_\beta\}
=
\{ \gamma|_{J_\beta} \textrm{ for } \gamma\in X_\beta\}
=\mathbb{A}^{v_\beta^0}.
$$
So there is a subspace $V$ of $\widetilde{X}_\beta$  of dimension $v_\beta^0$ such that $\{ \gamma|_{J_\beta} \textrm{ for } \gamma\in V \}$ has full dimension $v_\beta^0$. Then $V\cap X_\beta^\perp=0$ and $V+X_\beta^\perp=\widetilde{X}_\beta$, so $\dim X_\beta^\perp=v_\beta^\perp$. 

Next we shall show that $X_\beta^\perp\subseteq W_\beta'^\perp\oplus \bigoplus_{j=1}^t  X^\perp_{\alpha_{r'_j}}$, 
where $W_\beta'^\perp\cong \mathbb{A}^{w_\beta'+\sum_{h: s(h)=\beta} v_{t(h)}^0-v_\beta^0}$ consists of all vectors in $\mathbb{A}^{w'_\beta+\sum_{h:\;  s(h)=\beta } w_{t(h)}}$
whose $j$-th entry are 0 for $j\in J_\beta$ or $j\notin M_\beta$. 
Take an arbitrary element $P_\beta^*{\bf u} \in X_\beta^\perp$ 
for $u\in X_\beta$. We can write 
$${\bf u}=\begin{bmatrix} {\bf u}_0\\ {\bf u}_1\\ \vdots \\ {\bf u}_t \end{bmatrix}\; 
\textrm{ where ${\bf u}_0\in W'_\beta$, ${\bf u}_i\in X_{\alpha_{r'_i}}$ for $i=1,\dots,t$. }$$
Then
$$
P_\beta^*{\bf u}
=\begin{bmatrix}{\bf u}_0\\ 
\widetilde{P}_{\alpha_{r'_1}} {\bf u}_1\\ 
\vdots \\ 
\widetilde{P}_{\alpha_{r'_t}}{\bf u}_t 
\end{bmatrix}
={\bf v}_1+{\bf v}_2, 
\textrm{ where } 
{\bf v_1}=\begin{bmatrix}{\bf u}_0\\ \begin{bmatrix} {\bf I} &0\\ 0&0\end{bmatrix} {\bf u}_1\\ \vdots \\ \begin{bmatrix} {\bf I} &0\\ 0&0\end{bmatrix}  {\bf u}_t \end{bmatrix},\;
{\bf v}_2
=\begin{bmatrix} 0\\ \begin{bmatrix} 0 &0\\ -P_{\alpha_{r'_1}}&{\bf I}\end{bmatrix} {\bf u}_1\\ \vdots \\ \begin{bmatrix} 0 &0\\ -P_{\alpha_{r'_t}}&{\bf I}\end{bmatrix}  {\bf u}_t \end{bmatrix}
$$ 
Note that 
${\bf v}_1\in W_\beta'^{\perp}$
and 
${\bf v}_2\in {\bf 0}\oplus \bigoplus_{j=1}^t  X^\perp_{\alpha_{r'_j}}$. 
(Indeed, the second one is immediate from the definition of ${\bf v}_2$. To see the first one:
if $j\notin M_\beta$, the $j$-th entry of ${\bf v}_1$  is 0 by the definition of ${\bf v}_1$; if $j\in J_\beta$,  the $j$-th entry of ${\bf v}_2$ is 0 and the $j$-th entry of ${\bf v}_1+{\bf v}_2$ is 0  since $P^*u|_{J_\beta}={\bf 0}$, which implies that the $j$-th entry of ${\bf v}_1$ is 0). 
This proves $X_\beta^\perp\subseteq W_\beta'^\perp\oplus \bigoplus_{j=1}^t  X^\perp_{\alpha_{r'_j}}$.

Similar as in (1), we see that $\varphi$ is a morphism.

\medskip
\noindent (4) To show $\varphi$ is an isomorphism, we construct $\varphi^{-1}(q^0,(q^\perp,X_i^\perp))=(q,X_i)$ as follows: 
we let $q=\psi^{-1}(q^0,q^\perp)$, and for $\alpha\in I_0$, $\beta\in I_1$, let
$$\aligned
&X_\alpha
:= (\text{ column space of } A_\alpha({q^0}))+X_\alpha^\perp
= (\text{ column space of } \begin{bmatrix}{\bf I}\\ P_\alpha\end{bmatrix})+X_\alpha^\perp\\
&X_\beta
:=(\text{ column space of } A_\beta({q^0})+(P_\beta^*)^{-1}X_\beta^\perp\\
\endaligned
$$
It is routine to check that $\varphi$ and $\varphi^{-1}$ are indeed inverse to each other.

\medskip
\noindent (5) To check that the diagram commutes: 
$$\psi\circ\pi(q,X_i)=\psi(q)=(q^0,q^\perp)=(1\times\pi)(q^0,(q^\perp,X_i^\perp))=(1\times\pi)\circ\varphi(q,X_i)$$

\medskip
\noindent (6) To check $\psi(p)=(p,0)$: let $q=p$, and use the same assumption as above, we need to show $q^0=q$ and $q^\perp=0$. For $\alpha \in I_0$, since $A_\alpha(q)$ has rank $v_\alpha^0$, all other rows are linear combinations of its first $v_\alpha^0$ rows, that is, $A_\alpha(q)=\begin{bmatrix} A_\alpha(q)_{[1,\dots,v_\alpha^0;-]}\\  P'  A_\alpha(q)_{[1,\dots,v_\alpha^0;-]}\\  \end{bmatrix}$ for some matrix $P'$. Comparing with \eqref{AJBI} we see that $P'=P_\alpha$. Thus $Z_\alpha'=P_\alpha Z_\alpha$, $C_{r_j}''=P_\alpha C_{r_j}$, $C_{r_j}'''=P_\alpha C_{r_j}'$. Argue similarly for $A_\beta(q)$. We then conclude that $q=q^0$. Next, note that $\widetilde{A}_\alpha(q)=\begin{bmatrix} Z_\alpha&C_{r_1}&0&\cdots &C_{r_s}&0\\ 0&0&0&\cdots&0&0\end{bmatrix}$, so $A_\alpha(q)^\perp=0$, and similarly $A_\beta(q)^\perp=0$. Thus $q^\perp=0$.

The fact that $\psi({\bf E}^\circ_{u,w}\cap U)=U_0\times ({\bf E}^\circ_{u^\perp,w^\perp}\cap U^\perp)$, follows from \eqref{eq:rank Aqperp} and \eqref{eq:rank Bqperp}.
\end{proof}

\subsection{Proof of trivial local systems} \label{subsection:trivial local system}
It is the same as the proof of  \cite[Theorem 5.1]{L}. We reproduce it here for the readers' convenience. As pointed out in \cite[Remark 5.5]{L}, for the proof to work, it is essential to have the algebraic version, instead of the analytic version, of the Transversal Slice Theorem (Lemma \ref{lem:transversal slice}).
 
\begin{lemma}\label{lem:VXX}
Given $f:X\to Y$, $Y_a,L_a, n_a$ and the decomposition in \eqref{eq:decomp thm}. 
 Let $V$ be a nonsingular variety. Consider the following Cartesian diagram
$$\xymatrix{ V\times X\ar[r]^-{p'}\ar[d]_{1\times f}& X\ar[d]^f\\ V\times Y\ar[r]^-p&Y }$$
where $p$ and $p'$ are the natural projections. Define the pullback $\tilde{L}_a := p^*L_a$. Then
$$(1\times f)_* IC_{V\times X} \cong \bigoplus_{a} IC_{V\times\overline{Y_a}}(\tilde{L}_a)[n_a]$$
\end{lemma}
\begin{proof}
Same as \cite[Lemma 5.4]{L}. 
\end{proof}

\begin{proof}[Proof of Theorem \ref{thm:trivial local system}]
Let $IC_Z(L)[n]$ be a direct summand that appears in the decomposition of 
$\pi_*(IC_{\tilde{\mathcal{F}}_{v,w}})$. 
Take a general point $p$ of $Z$ is in ${\bf E}^\circ_{v^0,w}$ and apply Lemma \ref{lem:transversal slice}. Then $U^0$ is a Zariski open subset of ${\bf E}^\circ_{v^0,w}$. 
By uniqueness of the BBDG Decomposition, $IC_Z(L)[n]|_U$ is a direct summand of the decomposition of $\pi_*(IC_{\pi^{-1}U})$. By Lemma \ref{lem:VXX}, 
we have $IC_Z(L)[n]|_U\cong IC_{U^0\times\overline{Y_a}}(\tilde{L}_a)[n_a]$ for some $a$.   Thus $Z\cap U=U^0\times \overline{Y_a}$, $L\cong \tilde{L}_a$ on $Z\cap U$, and $n=n_a$. But  $Z\subseteq {\bf E}_{v^0,w}$ implies $Z\cap U\subseteq U^0\times\{0\}$. So we must have $Y_a=\{0\}$ and $Z\cap U=U^0\times\{0\}$, thus $Z={\bf E}_{v^0,w}$.  
Moreover, $L_a$ is the trivial local system $\mathbb{Q}_0$. So the local system $\tilde{L}_a$ on $U^0$ is also trivial since it is the pullback of $L_a$ under the map $U^0\times \{0\}\to \{0\}$. Then $L$ is trivial on $Z\cap U$, and we see that
$IC_Z(L)[n]\cong IC_{{\bf E}_{v^0,w}}[n]$. 
Let $v'=v^0$. Then $v' \le v$ as seen in Lemma \ref{lem:transversal slice}, and $(v',w)$ is $l$-dominant because 
${\bf E}^\circ_{v^0,w}$ is nonempty. 
\end{proof}

\section{Dual canonical basis}

The identification between dual canonical basis and triangular basis is explained to us by Fan Qin using results in \cite{Qin, KQ}. 

\subsection{Definition of $L(w)$, $M(w)$, $\mathbf{R}^\mathrm{finite}_t$, $\chi$}
Let $\mathcal{D}(\mathbf{E}_{w})$ be the bounded derived category of constructible sheaves of $\mathbb{Q}$-vector spaces on $\mathbf{E}_{w}$. For simplicity of notation, denote $IC_w(v) := IC_{\mathbf{E}_{v,w}}$. Define 
$\mathcal{P}_w := \{IC_w(v) \ | \  v\in\D(w)\}$. 
Define a full subcategory $\mathcal{Q}_w$ of $\mathcal{D}(\mathbf{E}_{w})$ whose objects are
 finite direct sums of $IC_w(v)[k]$  for various $v\in\D(w)$, $k\in\mathbb{Z}$.
  
Define an abelian group $\mathcal{K}_w$ to be generated by isomorphism classes $(L)$ of objects of $\mathcal{Q}_w$ and quotient by the relations $(L)+(L')=(L'')$ whenever $L\oplus L'\cong L''$. By abuse of notation we denote $(L)$ as $L$. We can view $\mathcal{K}_w$ as a free $\mathbb{Z}[t^\pm]$-module with a basis $\mathcal{P}_w$, by defining $t^iL=L[i]$ for $i\in\mathbb{Z}$. 
The duality on $\mathcal{D}(\mathbf{E}_{w})$ induces the bar involution on $\mathcal{K}_w$ satisfying
$\overline{tL}=t^{-1}\overline{L},\; \overline{IC_w(v)}=IC_w(v)$. By Theorem \ref{thm:trivial local system}, for arbitrary $v,w$, 
\begin{equation}\label{pi=sum aIC}
\pi_w(v) :=\pi_*(IC_{\widetilde{\mathcal{F}}_{v,w}})
=\sum_{v'\in \D(w)}\sum_{d}a^d_{v,v';w}IC_w(v')[d]
=\sum_{v'\in \D(w)}a_{v,v';w}IC_w(v')
\end{equation}
where $a_{v,v';w} :=\sum_d a^d_{v,v';w}t^d$. By Lemma \ref{lem:transversal slice}, 
\begin{equation}\label{eq:a=a}
a_{v,v';w}=a_{v^\perp,v'^\perp;w^\perp} \textrm{ for any $v^0\in\D(w)$}.
\end{equation} 
Then $\{ \pi_w(v)\ | \  v\in\D(w)\}$ is also a  $\mathbb{Z}[t^\pm]$-basis for $\mathcal{K}_w$, and $a_{v,v;w}=1$ for $(v,w)\in \D$. 
Define the dual $\mathcal{K}_w^* := {\rm Hom}_{\mathbb{Z}[t^\pm]}(\mathcal{K}_w,\mathbb{Z}[t^\pm])$, which is a free  $\mathbb{Z}[t^\pm]$-module with a basis
$\{L_w(v)\ |\ v\in\D(w)\}$, the dual basis to $\mathcal{P}_w$,
that is, 
$\langle L_w(v), IC_w(v')\rangle = \delta_{v,v'}$. The pairing $\langle-,-\rangle:\mathcal{K}_w^*\times \mathcal{K}_w\to \mathbb{Z}[t^\pm]$ satisfies the condition 
$\langle L,C[1]\rangle=\langle L[-1],C\rangle=t\langle L,C\rangle$, for all $L\in \mathcal{K}_w^*, C\in \mathcal{K}_w$. 
Then $\langle L_w(v), IC_w(v')\rangle = \delta_{v,v'}=\delta_{v^\perp,v'^\perp}=\langle f_{w^\perp}(v^\perp),IC_{w^\perp}(v'^{\perp})\rangle  \textrm{ for any $v^0\le v$}$, 
where $v^0$ is used to define $w^\perp=w-C_qv^0, v^\perp=v-v^0, v'^\perp=v'-v^0$. 
Define
$$\mathbf{R}_t := \{(f_w)\in \prod_w \mathcal{K}_w^* \ | \ \langle f_w,IC_w(v)\rangle=\langle f_{w^\perp}, IC_{w^\perp}(v^\perp)\rangle \textrm{ whenever }v\in\D(w), v^0\le v\}.$$
For $(f_w)\in \prod_w \mathcal{K}_w^*$, define $c_{wv}\in\mathbb{Z}[t^\pm]$ to satisfy
 $f_w=\sum_{v\in \D(w)} c_{wv}L_w(v)$. Then $c_{wv}=\langle f_w,IC_w(v)\rangle$, and 
$ (f_w)\in \mathbf{R}_t \textrm{ if and only if }c_{wv}=c_{w-C_qv,0} \textrm{ for every $l$-dominant pair $(w,v)$}$. 
So an element in $\mathbf{R}_t$ is uniquely determined by $\{c_{w0}\}_w$. 

Define $L(w)\in\mathbf{R}_t$ to be induced by $L_w(0)$,
that is, for any $(v',w')\in\D$, 
$\langle L(w),IC_{w'}(v')\rangle
=\delta_{w,w'-C_qv'}$. 
Define $\mathbf{R}^\mathrm{finite}_t$ to be the $\mathbb{Z}[t^\pm]$-submodule of $\mathbf{R}_t$ with the basis $\{L(w)\}_w$.

Define $\{M_w(v)\ | \ v\in\D(w)\}\in \mathcal{K}_w^*$ to be the functional 
$$(L)\mapsto \sum_k t^{\dim {\bf E}^\circ_{v,w}-k} \dim H^k(i^!_{x_{v,w}}L)$$
where $x_{v,w}$ is a point in ${\bf E}^\circ_{v,w}$, and $i_{x_{v,w}}: x_{v,w}\to {\bf E}_{w}$ is the natural embedding. 

Define $M(w)\in \mathbf{R}_t$ to be induced by $M_w(0)$, that is, denoting $i_0: \{0\}\to {\bf E}_w$ to be the natural embedding, then for any $(v',w')\in \D$, 
{\small
\begin{equation}\label{MIC}
\aligned
&\langle M(w),IC_{w'}(v')\rangle
\\
&=\begin{cases}
&\langle M_w(0),IC_{w}(v)\rangle
=\sum\limits_k t^{-k} \dim H^k(i^!_0IC_w(v)), 
\textrm{ if $\exists v\in\mathbb{Z}_{\ge0}^{n}: w-w'=C_q(v-v')$;}\\
&0, \textrm{ otherwise}.\\
\end{cases}
\endaligned
\end{equation}
}

Similar to \cite[\S7]{L}, 
denote a partial order
$w'\le_{\rm w} w \Leftrightarrow w'-w=C_qu$ for some $u \ge 0$ 
and denote 
$w'<_{\rm w} w$ if $w'\le_{\rm w} w$ and $w'\neq w$.
Then $\{M(w)\}_w$ is a basis of  $\mathbf{R}^\mathrm{finite}_t$, and  
\begin{equation}\label{L=sum M}
L(w)=\sum_{w''\le_{\rm w} w}b'_{ww''}M(w'')\in M(w)+\sum_{w''<_{\rm w} w}t^{-1}\mathbb{Z}[t^{-1}]M(w'').
\end{equation}

\medskip

Define a map 
\begin{equation}\label{eq:phi}
\Phi: \mathbb{Z}^{2n}\mapsto\mathbb{Z}^{2n},\quad \Phi(w)= \sum_{\beta \in I_1} (w_\beta-w'_\beta) e_\beta +\sum_{\alpha \in I_0}(w_\alpha'-w_\alpha)e_\alpha
\end{equation}
Define a map $\chi: \mathbf{R}^\mathrm{finite}_t\to \mathcal{T}$ by defining it on the basis $\{M(w)\}$ (note that we do not require $v\in\D(w)$): 
\begin{equation}\label{def_chiMw}
\aligned
\chi(M(w))
&=\sum_{v\in\mathbb{Z}_{\ge0}^{n}}   \big(\langle M(w),\pi_w(v)\rangle_{t\to\v^\delta}\big)   X^{\Phi(w)+\tilde{B}v}\\
&=\sum_{v\in\mathbb{Z}_{\ge0}^{n}}  \sum_k \v^{-k}\dim H^k(i_0^! \pi_w(v))   X^{\Phi(w)+\tilde{B}v}
\endaligned
\end{equation}
and extend it by the rules $\chi(L_1+L_2)=\chi(L_1)+\chi(L_2)$ and $\chi(tL)=\v^\delta\chi(L)$, 
The following is shown in \cite[\S7.3]{L}:
\begin{equation}\label{chiLw}
\chi(L(w))=\sum_{v\in\mathbb{Z}_{\ge0}^{n}} a_{v,0;w}(\v)X^{\Phi(w)+\tilde{B}v}
\end{equation}

\medskip

\subsection{$\{\overline{\chi(M(w))}\}$ and the standard monomial basis} 
All cohomology groups in this paper are taken with coefficient $\mathbb{Q}$, thus we write $\dim H^i(X)$ for $\dim H^i(X,\mathbb{Q})$.

Denote $P_\v(X):=\sum_i\dim H^i(X) \v^i$. 
For $n,k\in\mathbb{Z}_{\ge0}$, define the $q$-binomial coefficient as follows (where $q$ is replaced by $\v$).
$$[n] := \frac{\v^n-\v^{-n}}{\v-\v^{-1}},\quad 
{n \brack k} := \frac{[n] [n-1]\cdots [n-k+1]}{[k] [k-1] \cdots [1]} .$$
The cohomology groups of a (complex) Grassmannian is well-known; from which we have  
$$P_\v(Gr(k,n))
=\v^{\dim_\mathbb{C}Gr(k,n)} {n \brack k}
=\v^{k(n-k)} {n \brack k}
$$
The following is also well-known. 
\begin{lemma} \label{lemma:Grassmannian bundle}
Let $V$ be a complex vector bundle of rank $n$ over a complex variety $X$, $1\le d\le n-1$, and $E=Gr_X(d,V)$ be the associated Grassmannian bundle over $X$, $F\cong Gr(d,n)$ be a fiber. Then 
$$P_\v(E,\mathbb{Q})=P_\v(F,\mathbb{Q})  \cdot P_\v(X,\mathbb{Q}).$$  
\end{lemma}
\begin{proof}
It follows the Leray–Hirsch theorem \cite[\S4.D]{Hatcher} and
the surjectivity of $H^*(E)\to H^*(F)$. For the latter, recall that $A^*(X)$ is the Chow ring of $X$ and $c_k(E)$ (for $k=1,\dots,n$) are the Chern classes of $E$. 
By
\cite[Example 14.6.6]{Fulton},
$$
\aligned
&A^*(E)=A^*(X)[a_1,\dots,a_d,b_1,\dots,b_{n-d}]/ (\sum_{i=0}^k a_ib_{k-i}-c_k(E))_{k=1,\dots,n},\\
&A^*(F)=\mathbb{Q}[a_1,\dots,a_d,b_1,\dots,b_{n-d}]/ (\sum_{i=0}^k a_ib_{k-i}-\delta_{k,0})_{k=1,\dots,n}
\endaligned
$$
thus the restriction $A^k(E)\to A^k(F)$ is surjective for any $k$. So the composition $A^k(E)\to H^{2k}(E)\to H^{2k}(F)\cong A^k(F)$ is surjective. Thus $H^{2k}(E)\to H^{2k}(F)$ is surjective. (Alternatively, one can prove an analog of \cite[Example 14.6.6]{Fulton} for $H^*$ and not use Chow rings.)
\end{proof}

\begin{lemma}
{\rm(a)} We have
$$
\aligned
P_\v(\mathcal{F}_{v,w})
&=\prod_{\beta\in I_1}P_\v(Gr(v_\beta,w_\beta'+\sum_{h:s(h)=\beta} v_{t(h)}))
\cdot
\prod_{\alpha\in I_0} P_\v(Gr(v_\alpha,w_\alpha))\\
&=\v^{d_{v,w}} \prod_{\beta\in I_1}  
{w_\beta'+\sum_{h:s(h)=\beta} v_{t(h)} \brack v_\beta}
\cdot
\prod_{\alpha\in I_0} {w_\alpha \brack v_\alpha}\\
\endaligned
$$ 

{\rm(b)} Recall that 
$d_{v,w}=\dim \mathcal{F}_{v,w}$, $\tilde{d}_{v,w}=\dim \widetilde{\mathcal{F}}_{v,w}$. Then
\begin{equation}\label{chiMw}
\aligned
\chi(M(w))
&=\sum_{v\in\mathbb{Z}_{\ge0}^{n}}
\v^{d_{v,w}-\tilde{d}_{v,w}}
\prod_{\beta\in I_1}  
{w_\beta'+\sum_{h:s(h)=\beta} v_{t(h)} \brack v_\beta}
\cdot
\prod_{\alpha\in I_0} {w_\alpha \brack v_\alpha}
X^{\Phi(w)+\tilde{B}v}\\
&=\sum_{v\in\mathbb{Z}_{\ge0}^{n}}   
\v^{-\tilde{d}_{v,w} }  P_\v(\mathcal{F}_{v,w})  X^{\Phi(w)+\tilde{B}v}\\
\endaligned
\end{equation}
\end{lemma}

\begin{proof}
(a) follows from iteratively applying Lemma \ref{lemma:Grassmannian bundle}. 

(b) Denote the embedding $i_0:\{0\}\to \myE_{v,w}$. 
Then
\begin{equation}\label{Mpi}
\aligned
&\langle M(w),\pi_{w}(v)\rangle
=\v^{2d_{v,w}-\tilde{d}_{v,w}}\sum_k\dim t^{-k}H^{k}( { {\mathcal{F}}_{v,w} } )\
=\v^{2d_{v,w}-\tilde{d}_{v,w}}
\dim P_{\v^{-1}} ( { {\mathcal{F}}_{v,w} })\\
&
\stackrel{\rm(a)}{=}\v^{2d_{v,w}-\tilde{d}_{v,w}}
\v^{-\dim_\mathbb{C} \mathcal{F}_{v,w}} 
\prod_{\beta\in I_1}  
{w_\beta'+\sum_{h:s(h)=\beta} v_{t(h)} \brack v_\beta}
\cdot
\prod_{\alpha\in I_0} {w_\alpha \brack v_\alpha}
\\
&=\v^{d_{v,w}-\tilde{d}_{v,w}}
\prod_{\beta\in I_1}  
{w_\beta'+\sum_{h:s(h)=\beta} v_{t(h)} \brack v_\beta}
\cdot
\prod_{\alpha\in I_0} {w_\alpha \brack v_\alpha}
 =\v^{-\tilde{d}_{v,w}}
P_\v(\mathcal{F}_{v,w}). \\
\endaligned
\end{equation}
This together with \eqref{def_chiMw} imply \eqref{chiMw}.
\end{proof}

Without loss of generality, assume $I_1=\{1,\dots,k\}$, $I_0=\{k+1,\dots,n\}$. We denote $t'=\mu_{I_1}(t_0):=\mu_{1}\cdots\mu_{k}(t_0)$ and the attached seed is $(\Lambda',\tilde{B}',\tilde{X}')$. Thus 
$\tilde{B}'=\begin{bmatrix}-B\\J\end{bmatrix}$ where $J$ is the diagonal matrix with the first $k$ diagonal entries being $-1$ and the rest being $1$,
$\tilde{\Lambda}'=\begin{bmatrix}0&-J\\J&-B\end{bmatrix}$.  The standard monomial basis elements with respect to this initial seed is
$$
 E_a(t') = \v^{v'(a)} 
 X^{\sum_{i=n+1}^{2n}a_ie_i}
\prod_{i=n,\dots,1} \Big(X_i(t')^{[a_i]_+}X_i'(t')^{[-a_i]_+} 
\Big), \textrm{ for $a\in\mathbb{Z}^{2n}$.}
$$
where the index $i$ runs decreasingly from $n$ to $1$,  $X_i'(t')=\mu_i(X_i(t'))$, and $v'(a)\in\mathbb{Z}$ is determined by the condition that the leading term of $E_a(t')$ in seed $t'$ is bar-invariant. 
Then 
$X'_i(t')=X_i$ for $i\in I_1$, 
$X'_i(t')=X(t')^{-e_i+[b'_i]_+}+X(t')^{-e_i+[-b'_i]_+}$ for $i\in I_0$. Note that in the above expression of $E_a(t')$, the two factors in $X_i(t')^{[a_i]_+}X_i'(t')^{[-a_i]_+}$ can be swapped because one of  them is $1$.

We generalize the above expression into $E^*_a$ as follows: for $1\le i\le n$, replace $[-a_i]_+$ by $w_i$, $[a_i]_+$ by $w'_i$, and define
$$
E^*_{w,a_{n+1},\dots,a_{2n}} = \v^{v'} 
 X^{\sum_{i=n+1}^{2n}a_ie_i}
\prod_{i=n,\dots,k+1} 
\Big(
X_i'(t')^{w_i}X_i(t')^{w'_i} 
\Big)
\cdot
\prod_{i=k,\dots,1} 
\Big(
X_i(t')^{w'_i}X_i'(t')^{w_i} 
\Big)
$$
where $v'$ is similar as above. Note a technical subtlety that in the first product (for $i=n,\dots,k+1$) we have to write the two factors of $X_i'(t')^{w_i}X_i(t')^{w'_i}$ in that specific order; if we swap the two factors, then the power of $\v$ in  Lemma \ref{chiMbar=M''} will be incorrect. 

Each $w\in\mathbb{Z}_{\ge0}^{2n}$ can be uniquely written as $w={}^fw+{}^\phi w$ where both ${}^fw$ and ${}^\phi w$ are in $\mathbb{Z}_{\ge0}^{2n}$ and 
${}^\phi w$ satisfies ${}^\phi w_i{}^\phi w'_i=0$ for all $1\le i\le n$. 
We have
\begin{equation}\label{M'=M'}
\textrm{ if $w={}^\phi w$, then  
$E^*_{w,a_{n+1},\dots,a_{2n}}= E_a(t')$.}
\end{equation} 

The following is easy to show so omit skip the proof.
\begin{lemma}\label{lem:Xt'u}
Assume $u_i\ge0$ for $i\in I_1$. Then
$$X(t')^{u}=\sum_v(\prod_{i\in I_1} { u_i \brack v_i})X^{\sum_{i\in I_1} (-u_ie_i+v_ib_i) + \sum_{i\notin I_1} u_ie_i}$$
\end{lemma}

\begin{lemma}\label{chiMbar=M''}
For any $w\in\mathbb{Z}_{\ge0}^{2n}$, we have 
$$E^*_{w,0,\dots,0}=\overline{\chi(M(w))}
=\sum_{v\in\mathbb{Z}_{\ge0}^{n}}   \v^{\tilde{d}_{v,w}-d_{v,w}} (\prod_{\alpha\in I_0}  {w_\alpha\brack v_\alpha} 
\prod_{\beta \in I_1} { w'_\beta+\sum_{\alpha \in I_0} b_{\alpha\beta}v_\alpha \brack v_\beta} ) 
X^{ \Phi(w) +\sum_{1\le i\le n}v_i b_i  }.
$$ 
where ${\Phi}$ is defined in \eqref{eq:phi}.
In particular, if $w={}^\phi w$, then $\overline{\chi(M(w))}$ is the standard basis element $E_{(w'_i-w_i)_{1\le i\le n}}(t')$ with the initial seed at $t'=\mu_{I_1}(t_0)$.
\end{lemma}
\begin{proof}
For $i\in I_0$,  we have $b'_i\in\mathbb{Z}_{\ge0}^m$, thus  
$$
\aligned
X_i'(t')^{w_i} X_i(t')^{w'_i}
&=  \ \big(X(t')^{-e_i+b'_i}+X(t')^{-e_i}\big)^{w_i} X(t')^{w'_ie_i}\\
&= 
\sum_{v_i} {w_i\brack v_i}X(t')^{-w_ie_i  + v_ib'_i} X(t')^{w'_ie_i} \\
&=
 \sum_{v_i} \v^{\Lambda'(-w_ie_i+v_ib'_i,w'_ie_i)}{w_i\brack v_i}X(t')^{(w'_i-w_i)e_i  + v_ib'_i}\\
&=
 \sum_{v_i} \v^{w'_iv_i}{w_i\brack v_i}X(t')^{(w'_i-w_i)e_i +  v_ib'_i}\\
\endaligned
$$ 

For $i\in I_1$,  
$X_i(t')^{w'_i}X_i'(t')^{w_i}=X(t')^{w'_ie_i} X_i^{w_i} $ 

By comparing the leading terms, it is easy to see that the power of $\v$ to normalize 
$E^*_w$ in $t'$ is the same as the one to normalize $E^*_w$ in $t$. 
Thus we ignore factor of powers of $\v$ in the following computation, and denote $f\sim \v^if$ for $i\in\mathbb{Z}$. 
$$\aligned
&E^*_{w,0,\dots,0}
\sim (\prod_{i=n,\dots,k+1} \sum_{v_i} \v^{w'_iv_i}{w_i\brack v_i}X(t')^{(w'_i-w_i)e_i + v_ib'_i}) \cdot (\prod_{i=k,\dots,1} X(t')^{w'_ie_i} \ X_i^{w_i})\\
&\sim \sum_{v_n,\dots,v_{k+1}} \v^{p_1}(\prod_{i=n,\dots,k+1} {w_i\brack v_i}) X(t')^{ \sum_{i=k+1}^n \big( (w'_i-w_i)e_i+v_ib'_i\big)+ \sum_{i=1}^k w'_ie_i} \ 
\prod_{i=k,\dots,1} X_i^{w_i}\\
&\sim \sum_{v_n,\dots,v_{k+1}} \v^{p_1}(\prod_{i=n,\dots,k+1}  {w_i\brack v_i}) X(t')^{ \sum_{i\in I_1} (w'_i+\sum_{j\in I_0} b_{ji}v_j)e_i +\sum_{i\in I_0}(w_i'-w_i)e_i +\sum_{i\in I_0}v_ie_{n+i} } 
\ 
\prod_{i=k,\dots,1} X_i^{w_i}\\
&\sim \sum_{v} \v^{p_1}( \prod_{\alpha\in I_0}  {w_\alpha\brack v_\alpha} 
\prod_{\beta \in I_1} { w'_\beta+\sum_{\alpha \in I_0} b_{\alpha\beta}v_\alpha \brack v_\beta} ) \\
&\quad \cdot X^{ \sum_{\beta \in I_1} ((-w'_\beta-\sum_{\alpha\in I_0}  v_\alpha b_{\beta\alpha})e_\beta +v_\beta b_\beta ) +\sum_{\alpha \in I_0}(w_\alpha'-w_\alpha)e_\alpha +\sum_{\alpha\in I_0}v_\alpha e_{n+\alpha} } 
\ 
\prod_{\beta=k,\dots,1} X_\beta^{w_\beta} 
\textrm{ (Lemma \ref{lem:Xt'u}) }\\
&\sim \sum_{v} \v^{p_1}( \prod_{\alpha\in I_0}  {w_\alpha\brack v_\alpha} 
\prod_{\beta \in I_1} { w'_\beta+\sum_{\alpha \in I_0} b_{\alpha\beta}v_\alpha \brack v_\beta} ) \\
&\quad \cdot X^{ \sum_{\beta \in I_1} (-w'_\beta e_\beta ) +\sum_{\alpha \in I_0}(w_\alpha'-w_\alpha)e_\alpha +\sum_{1\le i\le n}v_i b_i  } 
\ 
\prod_{\beta=k,\dots,1} X_\beta^{w_\beta}  
\textrm{ (used $v_\alpha e_{n+\alpha}-\sum_{\beta\in I_1} v_\alpha b_{\beta\alpha}e_\beta  = v_\alpha b_\alpha$) } \\
&\sim \sum_{v} \v^{p_2} 
(\prod_{\alpha\in I_0}  {w_\alpha\brack v_\alpha} 
\prod_{\beta \in I_1} { w'_\beta+\sum_{\alpha \in I_0} b_{\alpha\beta}v_\alpha \brack v_\beta} ) 
 X^{ \sum_{\beta \in I_1} (w_\beta-w'_\beta) e_\beta +\sum_{\alpha \in I_0}(w_\alpha'-w_\alpha)e_\alpha +\sum_{1\le i\le n}v_i b_i  }\\
\endaligned
$$ 
where 
$p_1=\sum_{\alpha\in I_0} w'_\alpha v_\alpha$, 
$p_2=\sum_{\alpha\in I_0} w'_\alpha v_\alpha+\sum_{\beta\in I_1} w_\beta v_\beta=\tilde{d}_{v,w}-d_{v,w}$. 
Since the last expression is already normalized, it gives a correct formula for $E_{w,0,\dots,0}^*$. This completes the proof.  
\end{proof}

\begin{lemma}\label{M'' in M'}
Let $w\in\mathbb{Z}_{\ge0}^{2n}$, $a=(w'_1-w_1,\dots,w'_n-w_n,a_{n+1},\dots,a_{2n})\in\mathbb{Z}^{2n}$. 

{\rm(a)} If $\min(w_\alpha,w'_\alpha)>0$ for some $\alpha\in I_0$, then
$$E^*_{w,a_{n+1},\dots,a_{2n}} = E^*_{w^{(1)},a_{n+1},\dots,a_{2n}}+ \v^p E^*_{w^{(2)},a'_{n+1},\dots,a_{n+\alpha}+1,\dots,a_{2n}}$$ 
where $p$ is a positive integer, $w^{(1)}$ is obtained from $w$ by $w_\alpha\mapsto w_\alpha-1$ and $w'_\alpha\mapsto w'_\alpha-1$; 
$w^{(2)}$ is obtained from $w$ by $w_\alpha\mapsto w_\alpha-1$,  $w'_\alpha\mapsto w'_\alpha-1$, $w'_\beta \mapsto w'_\beta+b'_{\beta \alpha}$. 

{\rm(b)} If $\min(w_\beta,w'_\beta)>0$ for some $\beta\in I_0$, then
$$E^*_{w,a_{n+1},\dots,a_{2n}} = E^*_{w^{(1)},a_{n+1},\dots,a_{2n}}+ \v^p E^*_{w^{(2)},a'_{n+1},\dots,a_{n+\beta}+1,\dots,a_{2n}}$$ 
where $p$ is a positive integer, $w^{(1)}$ is obtained from $w$ by $w_\beta\mapsto w_\beta-1$ and $w'_\beta\mapsto w'_\beta-1$; 
$w^{(2)}$ is obtained from $w$ by $w_\beta\mapsto w_\beta-1$,  $w'_\beta\mapsto w'_\beta-1$, $w'_\alpha \mapsto w'_\alpha-b'_{\alpha\beta}$. 

{\rm(c)} We have 
$$E^*_{w,a_{n+1},\dots,a_{2n}}\in E_a(t')+ \sum_{r(b)\le r(a)} \v\mathbb{Z}[\v]E_b(t')$$ 

\end{lemma}
\begin{proof}
(a) If $\min(w_\alpha,w'_\alpha)>0$ for $\alpha\in I_0$, then 
$$\aligned
&E^*_{w,a_{n+1},\dots,a_{2n}}
\sim
X^{\sum_{i=n+1}^{2n}a_ie_i}
\prod_{i=n,\dots,\alpha+1} 
\Big(
X_i'(t')^{w_i}  X_i(t')^{w'_i} 
\Big)
\cdot
X_\alpha'(t')^{w_\alpha-1} 
\cdot 
\big(X_\alpha'(t') X_\alpha(t')\big)
\\
&\quad\quad\quad\quad\quad\quad 
\cdot
X_\alpha(t')^{w'_\alpha-1} 
\cdot
\prod_{i=\alpha-1,\dots,k+1} 
\Big(
X_i'(t')^{w_i}  X_i(t')^{w'_i} 
\Big)
\prod_{i=k,\dots,1} 
\Big(
X_i(t')^{w'_i}X_i'(t')^{w_i} 
\Big)
\\
&\sim
X^{\sum_{i=n+1}^{2n}a_ie_i}
\prod_{i=n,\dots,\alpha+1} 
\Big(
X_i'(t')^{w_i}  X_i(t')^{w'_i} 
\Big)
\cdot
X_\alpha'(t')^{w_\alpha-1} 
\cdot 
\big(1+ \v X(t')^{b'_\alpha}\big)
\\
&\quad\quad 
\cdot
X_\alpha(t')^{w'_\alpha-1} 
\cdot
\prod_{i=\alpha-1,\dots,k+1} 
\Big(
X_i'(t')^{w_i}  X_i(t')^{w'_i} 
\Big)
\prod_{i=k,\dots,1} 
\Big(
X_i(t')^{w'_i}X_i'(t')^{w_i} 
\Big)
\\
&\sim \v^{p_1} E^*_{w^{(1)},a_{n+1},\dots,a_{2n}}+\v^{p_2} E^*_{w^{(2)},a_{n+1},\dots,a_{n+\alpha}+1,\dots,a_{2n}}
\\
\endaligned
$$
where 
$w^{(1)}$ is obtained from $w$ by $w_\alpha\mapsto w_\alpha-1$ and $w'_\alpha\mapsto w'_\alpha-1$; 
$w^{(2)}$ is obtained from $w$ by $w_\alpha\mapsto w_\alpha-1$,  $w'_\alpha\mapsto w'_\alpha-1$, $w'_\beta \mapsto w'_\beta+b'_{\beta \alpha}$. 
Note that $(w^{(2)},a_{n+1},\dots,a_{n+\alpha}+1,\dots,a_{2n})$ corresponds to $b=(a_1+b'_{1\alpha},\dots,a_k+b'_{k\alpha},a_{k+1},\dots,a_n,a_{n+1},\dots,a_{n+\alpha}+1,\dots,a_{2n})$, 
which satisfies $b\prec a$ or, unpleasantly, $r(b)=r(a)$. If the latter case happens, at least we see that $\sum_i w_i$ strictly decreases. 

Next, we show that $p>0$. Indeed, note that $e_\alpha'=-e_\alpha$ for $\alpha\in I_0$, $e_\beta'=-e_\beta+(-b_\beta')$ for $\beta\in I_0$, and 
$p_1=\Lambda'\big(\sum_{n+1}^{2n}a_ie_i,w_ne'_n,w_n'e_n,\dots,(w_\alpha-1)e'_\alpha$, $(w'_\alpha-1)e_\alpha,\dots$, $w_{k+1}e'_{k+1},w'_{k+1}e_{k+1},w'_ke_k,w_ke'_k,\dots,w'_1e_1,w_1e'_1\big)$, 
while $p_2=1+\Lambda'\big(\sum_{n+1}^{2n}a_ie_i$, $w_ne'_n$, $w_n'e_n$, $\dots$, $(w_\alpha-1)e'_\alpha,b'_\alpha,(w'_\alpha-1)e_\alpha,\dots,w_{k+1}e'_{k+1},w'_{k+1}e_{k+1},w'_ke_k,w_ke'_k,\dots,w'_1e_1,w_1e'_1\big)$.
It is routine to verify that 
$$p=p_2-p_1= w_\alpha+w'_\alpha-1+\sum_{\beta\in I_1} w_\beta(-b'_{\alpha\beta})\ge1$$
where the last inequality is because  $b'_{\alpha\beta}\le0$ for $\alpha\in I_0, \beta\in I_1$. 

(b) is proved similarly to (a). 

(c) is obtained by applying (a) and (b) recursively.
\end{proof}

\subsection{$\{\chi(L(w))\}$ and the triangular basis}

\begin{lemma}\label{chiL(w)=C}
For any $w\in\mathbb{Z}_{\ge0}^{2n}$, ${\chi(L(w))}=C_{\Phi(w)}$ is a triangular basis element. 
In particular, $\chi(L({}^\phi w))=\chi(L(w))$ (therefore $a_{v,0;w}=a_{v,0;{}^\phi w }$ for any $v\in\mathbb{Z}_{\ge0}^n$) and $\chi(L({}^f w))=1$. 
\end{lemma}
\begin{proof}
It follows from \eqref{chiLw} that $\chi(L(w))$ is bar-invariant. 
Meanwhile,  
$$
\aligned
&\chi(L(w))-E_{\Phi(w)}(t')\in \chi(L(w))-E^*_{w,0,\dots,0}+ \sum \v\mathbb{Z}[\v]E_b(t') \quad\text{ (by Lemma  \ref{M'' in M'}) }\\
&\quad =\chi(L(w))-\overline{\chi(M(w))}+ \sum \v\mathbb{Z}[\v]E_b(t')\quad\text{ (by Lemma  \ref{chiMbar=M''}) }\\
&\quad \subseteq \overline{\chi(M(w))}+\sum\v\mathbb{Z}[\v]\overline{\chi(M(w''))}-\overline{\chi(M(w))}+ \sum \v\mathbb{Z}[\v]E_b(t')\quad\text{ (by \eqref{L=sum M})}\\
&\quad = \sum\v\mathbb{Z}[\v] E^*_{w'',0,\dots,0}+\sum \v\mathbb{Z}[\v]E_b(t') \quad\text{ (by  Lemma \ref{chiMbar=M''})}\\
&\quad \subseteq \sum\v\mathbb{Z}[\v] (\sum \mathbb{Z}[\v] E_{b'} (t') +\sum \v\mathbb{Z}[\v]E_b(t')\\
&\quad=\sum \v\mathbb{Z}[\v]E_b(t')
\endaligned
$$ 
By \cite[Theorem 1.1]{BZ2}, $\chi(L(w))$ equals to a triangular basis element for the seed at $t'$. Since the triangular basis does not depend on the chosen acyclic seed (also proved in \cite{BZ2}), $\chi(L(w))$ is also a triangular basis element for initial seed at $t_0$.  
\end{proof}

\section{The proof of the main result}
We cite the following lemma in \cite{L}.

\begin{lemma}\label{BBDG property}
Let $f: Y\to X$ be a proper morphism between complex algebraic varieties, $Y$ be nonsingular,  let $0$ be a point in $X$. 
Let $d=\dim Y$, $Y_0=f^{-1}(0)$ and $d_0=\dim Y_0$. Write the BBDG decomposition in the form
\begin{equation}\label{decomposition 0}f_*IC_Y=\bigoplus_b IC_{0}^{\oplus s_{0,b}}[b]\oplus\bigoplus_{V,L,b} IC_{V}^{\oplus s_{V,L,b}}(L)[b]
\end{equation}
where $V\neq 0$ are subvarieties of $X$,  each $L$ is a local system on an open dense subset of $V$, $s_{0,b}, s_{V,L,b}\in\mathbb{Z}_{\ge0}$ are multiplicities of the corresponding $IC$-sheaves. Then  $\{s_{0,b}\}$ satisfy the following conditions:

{\rm i)}  $s_{0,b}=s_{0,-b}$ for every $b\in\mathbb{Z}$.

{\rm ii)} $s_{0,b}\ge s_{0,b+2}$ for  every $b\in\mathbb{Z}_{\ge 0}$.

{\rm iii)} $s_{0,b}=0$ if $|b|>2d_0-d$. In particular, if $2d_0<d$, then $s_{0,b}=0$ for all $b$.
\end{lemma}

\begin{proof}[Proof of Theorem \ref{main theorem}]
For the given $a=(a_i)\in\mathbb{Z}^{2n}$, define $w=(w_i,w'_i)_{1\le i\le n}\in\mathbb{Z}_{\ge0}^{2n}$ by $w'_i=[a_i]_+, w_i=[-a_i]_+$. Then $\Phi(w)=(a_1,\dots,a_n,0,\dots,0)$. 
By \eqref{chiLw}, $e_v= a_{v,0;w}$. 
Apply Lemma \ref{BBDG property} to the projection $\pi: \tilde{\mathcal{F}}_{v,w}\to \mathbf{E}_{v,w}$ defined in \eqref{pi}, we conclude that $e_v$ is symmetric, unimodal, and
$$\deg(e_v)=\deg a_{v,0;w}\le 2d_{v,w}-\tilde{d}_{v,w}=f(v)$$
by Proposition \ref{prop:fiber} (b).  
\end{proof}
\medskip

We end the paper with two examples. 

\begin{example}
(a) 
Consider the initial seed determined by $B=\begin{bmatrix}0&-3\\3&0\end{bmatrix}$. 

The corresponding quiver $\mathcal{Q}^{\rm op}$ is $1\triplerightarrow{} 2$. 
\medskip

We compute the triangle basis element at $a=(9,-4,0,0)$. The corresponding $w=(w_1,w'_1,w_2,w'_2)=(0,9,4,0)$. Then
\medskip

$E_{(9,-4,0,0)}=X^{(-3, -4, 0, 4)} + (\v^2+1+\v^{-2})X^{(-3, -1, 1, 4)} + (\v^2+1+\v^{-2})X^{(-3, 2, 2, 4)} + X^{(-3, 5, 3, 4)} + (\v^3+\v+\v^{-1}+\v^{-3})X^{(0, -4, 0, 3)} + (\v^2+1+\v^{-2})X^{(0, -1, 1, 3)} + (\v^4+\v^2+2+\v^{-2}+\v^{-4})X^{(3, -4, 0, 2)} + X^{(3, -1, 1, 2)} + (\v^3+\v+\v^{-1}+\v^{-3})X^{(6, -4, 0, 1)} + X^{(9, -4, 0, 0)}$, 
\medskip

$f(v)=-(9+v_1)v_1-(-4+v_2)v_2+ 3v_1v_2=-v_1^2+3v_1v_2-v_2^2-9v_1+4v_2$. 
\medskip

${\rm Supp}(C_a)=\{(0,0),((0,1),(0,2),(0,3),(0,4),(1,2),(1,3),(1,4),(2,4),(3,4)\}$. 

\begin{figure}[h]
\centering
\begin{minipage}{.45\textwidth}
\vspace{1.5cm}

\begin{center}
\includegraphics[width=5cm]{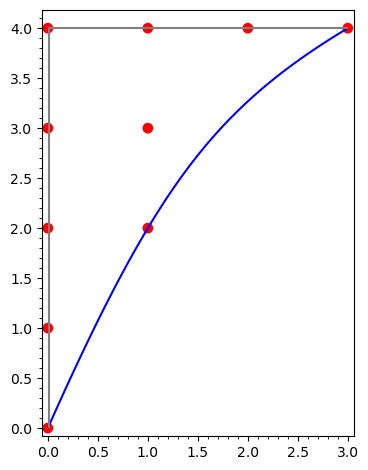}
\end{center}
\end{minipage}
\begin{minipage}{.5\textwidth}
\includegraphics[width=9cm,clip=true, trim = 100mm 0mm 30mm 0mm]{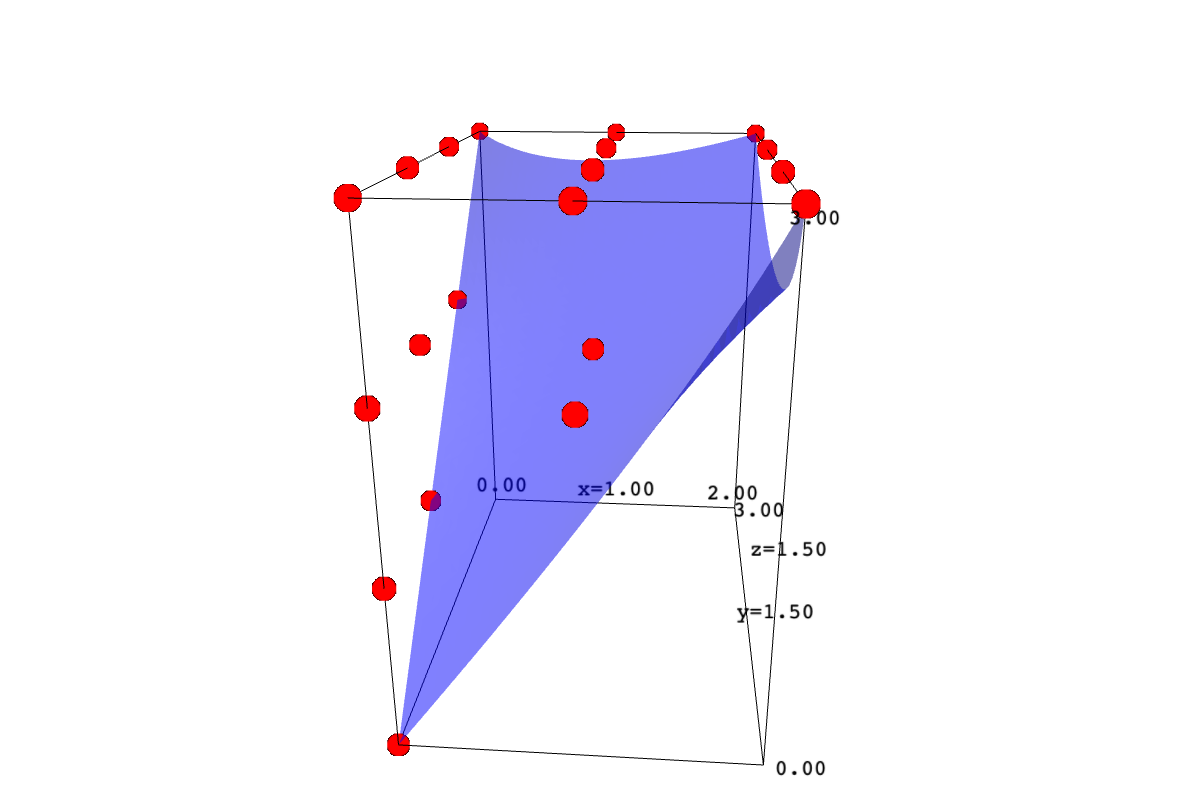}
\end{minipage}
\caption{Left: the support of $C_{(9,-4,0,0)}$ in (a). Right: the support of $C_{(4,3,-3,0,0,0)}$ in (b). The points in ${\rm Supp}(C_a)$ are in red color.  }
\label{fig:example}
\end{figure}
\medskip

See Figure \ref{fig:example} Left. The inequality $f(v)\ge0$ gives the upper left side of the blue curve.    This triangular basis element corresponds to $C[3,4]$ given in \cite[\S10.1(a)]{L}. Note that the $d$-vector $(3,4)$ corresponds to $g$-vector $(9,-4)$. 
\medskip

(b) 
Consider the initial seed determined by
$=\begin{bmatrix}0&0&-2\\ 0&0&-2\\ 2&2&0\end{bmatrix}$. 
\medskip

The corresponding quiver $\mathcal{Q}^{\rm op}$ is $1\doublerightarrow{}{}3\doubleleftarrow{}{}2$.
\medskip

We compute the triangle basis element at $a=(4,3,-3,0,0,0)$. The corresponding $w=(w_1,w'_1,w_2,w'_2,w_3,w'_3)=(0,4,0,3,3,0)$. Then
\medskip

$E_{(4,3,-3,0,0,0)}=X^{(-2, -3, -3, 0, 0, 3)} + (\v^2+1+\v^{-2})X^{(-2, -3, -1, 0, 1, 3)} + (\v+\v^{-1})*X^{(-2, -3, -1, 1, 0, 3)} + (\v^2+1+\v^{-2})X^{(-2, -3, 1, 0, 2, 3)} + (\v^3+2\v+2\v^{-1}+\v^{-3})X^{(-2, -3, 1, 1, 1, 3)} + X^{(-2, -3, 1, 2, 0, 3)} + X^{(-2, -3, 3, 0, 3, 3)} + (\v^3+2\v+2\v^{-1}+\v^{-3})X^{(-2, -3, 3, 1, 2, 3)} + (\v^2+1+\v^{-2})X^{(-2, -3, 3, 2, 1, 3)} + (\v+\v^{-1})X^{(-2, -3, 5, 1, 3, 3)} + (\v^2+1+\v^{-2})X^{(-2, -3, 5, 2, 2, 3)} + X^{(-2, -3, 7, 2, 3, 3)} + (\v^2+1+\v^{-2})X^{(0, -1, -3, 0, 0, 2)} + (\v^2+2+\v^{-2})X^{(0, -1, -1, 0, 1, 2)} + (\v+\v^{-1})X^{(0, -1, -1, 1, 0, 2)} + X^{(0, -1, 1, 0, 2, 2)} + (\v+\v^{-1})X^{(0, -1, 1, 1, 1, 2)} + (\v^2+1+\v^{-2})X^{(2, 1, -3, 0, 0, 1)} + X^{(2, 1, -1, 0, 1, 1)} + X^{(4, 3, -3, 0, 0, 0)}$, 
\medskip

$f(v)=-(4+v_1)v_1-(3+v_2)v_2-(-3+v_3)v_3+ 2v_1v_3+2v_2v_3=-v_1^2-v_2^2-v_3^2+2v_1v_3+2v_2v_3-4v_1-3v_2+3v_3$. 

\medskip

${\rm Supp}(C_a)=\{(0,0,0),  (0,0,1),  (0,0,2), (0,0,3), (0,1,1),(0,1,2), (0,1,3), (0,2,2), (0,2,3)$, $(0,3,3), (1,0,2), (1,0,3), (1,1,2),(1,1,3), (1,2,3), (1,3,3), (2,0,3), (2,1,3), (2,2,3), (2,3,3)\}$. 
\smallskip

See Figure \ref{fig:example} Right. The inequality $f(v)\ge0$ gives the upper region of the blue surface.    Note that the $d$-vector $(2,3,3)$ corresponds to $g$-vector $(4,3,-3)$. 
\medskip

\end{example}

\end{document}